\documentclass[11pt]{amsart}
\usepackage{amssymb,amsmath,amsfonts,amscd,euscript}
\usepackage{hyperref}

\newcommand{\nc}{\newcommand}

\numberwithin{equation}{section}
\newtheorem{thm}{Theorem}[section]
\newtheorem{prop}[thm]{Proposition}
\newtheorem{lem}[thm]{Lemma}
\newtheorem{cor}[thm]{Corollary}
\theoremstyle{remark}
\newtheorem{rem}[thm]{Remark}
\newtheorem{definition}[thm]{{\bf Definition}}
\newtheorem{example}[thm]{Example}

\nc{\gl}{\mathfrak{gl}}
\nc{\GL}{\mathfrak{GL}}
\nc{\g}{\mathfrak{g}}
\nc{\gh}{\widehat\g}
\nc{\h}{\mathfrak{h}}
\nc{\la}{\lambda}
\nc{\al}{\alpha }
\nc{\be}{\beta }
\nc{\ve}{\varepsilon }
\nc{\om}{\omega }

\nc{\ta}{\theta}
\nc{\ch}{{\mathop {\rm ch}}}
\nc{\Tr}{{\mathop {\rm Tr}\,}}
\nc{\Id}{{\mathop {\rm Id}}}
\nc{\ad}{{\mathop {\rm ad}}}
\nc{\bra}{\langle}
\nc{\ket}{\rangle}
\nc{\x}{{\bf x}}
\nc{\bm}{{\bf m}}
\nc{\bs}{{\bf s}}
\nc{\ba}{{\bf a}}
\nc{\bb}{{\bf b}}
\nc{\bk}{{\bf k}}
\nc{\bp}{{\bf p}}
\nc{\pa}{\partial}
\nc{\ld}{\ldots}
\nc{\cd}{\cdots}
\nc{\hk}{\hookrightarrow}
\nc{\T}{\otimes}
\nc{\gr}{\mathrm{gr}}
\nc{\ov}{\overline}

\nc{\cO}{\mathcal O}
\nc{\msl}{\mathfrak{sl}}
\nc{\mgl}{\mathfrak{gl}}
\nc{\U}{\mathrm U}
\nc{\V}{\EuScript V}
\nc{\cL}{\mathcal{L}}
\nc{\Res}{\mathrm{Res\ }}

\newcommand{\bK}{{\mathbb K}}

\newcommand{\bZ}{{\mathbb Z}}

\newcommand{\fh}{{\mathfrak h}}

\newcommand{\fg}{{\mathfrak g}}

\newcommand{\fn}{{\mathfrak n}}

\begin{document}

\title[Generalized Weyl modules and Macdonald polynomials]
{Generalized Weyl modules, alcove paths and Macdonald polynomials}

\author{Evgeny Feigin}
\address{Evgeny Feigin:\newline
Department of Mathematics,\newline
National Research University Higher School of Economics,\newline
Usacheva str. 6, 119048, Moscow, Russia,\newline
{\it and }\newline
Tamm Theory Division, Lebedev Physics Institute
}
\email{evgfeig@gmail.com}
\author{Ievgen Makedonskyi}
\address{Ievgen Makedonskyi:\newline
Department of Mathematics,\newline
National Research University Higher School of Economics,\newline
Usacheva str. 6, 119048, Moscow, Russia
}
\email{makedonskii\_e@mail.ru}

\begin{abstract}
Classical local Weyl modules for a simple Lie algebra are labeled by dominant weights. We generalize the definition to the case of
arbitrary weights and study the properties of the generalized modules. We prove that the representation theory of the generalized Weyl modules
can be described in terms of the alcove paths and the quantum Bruhat graph. We make use of the Orr-Shimozono formula
in order to prove that the $t=\infty$ specializations
of the nonsymmetric Macdonald polynomials are equal to the characters of certain generalized Weyl modules.
\end{abstract}

\maketitle

\section*{Introduction}
Let $\fg$ be a simple Lie algebra and let $E_\la(x,q,t)$ be the nonsymmetric Macdonald polynomials
attached to $\fg$ \cite{Ch1,O,M2}. These are polynomials in the (multi)variable $x$ with coefficients
being rational functions in $q$ and $t$; the parameter $\la$ is a weight from the weight lattice of the simple Lie algebra.
The symmetric Macdonald polynomials $P_\la(x,q,t)$ can be obtained from $E_\la(x,q,t)$ via certain symmetrization procedure
\cite{HHL}.
The polynomials $E_\la$ can be defined in two different ways: either as the eigenfunctions of certain commuting operators or
via the Cherednik inner product. They form a basis of the polynomial module of the double affine Hecke algebra.

The nonsymmetric Macdonald polynomials proved to play an important
role in representation theory: the specializations $E_\la(x,q,0)$ were identified with the characters
of the level one Demazure modules of the corresponding affine Kac-Moody Lie algebras (see \cite{S,I}).
It has been demonstrated recently (\cite{CO2,CF,OS}) that for anti-dominant weights $\la$ the specialization $t=\infty$ is also
very meaningful.
In particular, the functions $E_\la(x,q^{-1},\infty)$ turned out to be polynomials in $x,q$ with
nonnegative integer coefficients \cite{OS}; these polynomials were conjectured in \cite{CO2} to coincide with the PBW
twisted characters of the level one Demazure modules (see also \cite{CF,FM1,FM2}). One of the motivations of our paper is to 
categorify the Orr-Shimozono combinatorial construction. In particular, we are aimed at giving a
representation theoretic realization of the polynomials $E_\la(x,q^{-1},\infty)$. It turns out that much richer structure
is available. Namely, let us fix an anti-dominant weight $\la$ and let $W$ be the Weyl group of $\fg$.
We construct a family of modules $W_{\sigma(\la)}$, $\sigma\in W$, such that the characters of
$W_{\sigma(\la)}$ interpolate between $E_{\la}(x,q,0)$ and $E_{\la}(x,q^{-1},\infty)$.
The two main ingredients we need are the alcove path model
and the local Weyl modules. We note that there also exist global Weyl modules, but in this paper we only deal with
the local variant. So in what follows when we write the Weyl module(s) we mean the local version.

The classical Weyl modules $W(\la)$ are the $\fg\T\bK[t]$ modules labeled by dominant weights $\la$ (see \cite{CP,CL,FL1,FL2}).
These are cyclic modules
defined by generators and relations. In our paper we introduce the generalized Weyl modules $W_\mu$, depending
on an arbitrary weight $\mu$. Let $\la$ be an anti-dominant weight and let $\sigma\in W$.
\medskip

\noindent{\bf Definition}.
The generalized Weyl module $W_{\sigma(\la)}$ is a cyclic representation of the algebra
$\fn^{af}=\fg\T t\bK[t]\oplus\fn_+\T 1$ defined by the set of relations ($v$ is the cyclic vector):
\begin{gather*}
h\T t^k v=0 \text{ for all } h\in\fh, k>0;\\
(f_{\alpha}\otimes t) v=0, \ \alpha \in \sigma(\Delta_-)\cap \Delta_-;\\
(e_\alpha\otimes 1) v=0, \ \alpha \in \sigma(\Delta_-)\cap\Delta_+;\\
(f_{\sigma(\alpha)}\otimes t)^{-\langle \alpha^\vee, \lambda \rangle+1} v=0, \ \alpha \in \Delta_+, \sigma(\alpha) \in \Delta_-;\\
(e_{\sigma(\alpha)}\otimes 1)^{-\langle \alpha^\vee, \lambda \rangle+1} v=0,\ \alpha \in \Delta_+,\ \sigma(\alpha) \in \Delta_+.
\end{gather*}

We use the standard notation from the Lie theory, see Section 2 for details.
One sees from the definition that for an anti-dominant $\la$ we have the isomorphism of $\fn^{af}$ modules $W(w_0\la)\simeq W_{\la}$.
We prove the following theorem.

\medskip

\noindent{\bf Theorem A}. Let $\la$ be an anti-dominant weight, $\sigma\in W$. Then
\begin{enumerate}
\item $\dim W_{\sigma(\la)}=\dim W_\la$, $W_{\sigma(\la)}\simeq W_\la$ as $\fh$-modules..
\item ${\rm ch} W_{w_0\la}=w_0E_{\la}(x,q^{-1},\infty)$.
\item ${\rm ch} W_{\la}=E_{\la}(x,q,0)$.
\item For any $i=1,\dots,{\rm rk}(\fg)$ such that $\langle\la,\al_i^\vee\rangle < 0$ the module $W_{\sigma(\la)}$
can be decomposed into subquotients of the form $W_{\kappa(\la+\omega_i)}$, $\kappa\in W$. The subquotients are labeled
by certain alcove paths and the number of subquotients is equal to the dimension of the fundamental classical
Weyl module $W(\om_i)$.
\end{enumerate}
We note that the representation theoretic and geometric realizations of the polynomials $E_{\la}(x,q^{-1},\infty)$ for 
non anti-dominant weights can be found in \cite{FM3,FMO,Kat,FMK}. In particular, in \cite{Kat} the author
realizes the generalized Weyl modules as dual spaces of sections of line bundles on certain quotients
of semi-infinite Schubert varieties.  Also in the paper \cite{NNS} the authors study
a quantum analogue of the generalized Weyl modules -- the Demazure submodules of extremal weight modules. 

\medskip

The last part of Theorem A explains the importance of the third ingredient of the picture: the alcove paths model
(see \cite{GL,LP}). Namely, the $t=0$ and $t=\infty$ specializations of the nonsymmetric Macdonald polynomials enjoy the
combinatorial realization in terms of quantum alcove paths in the affine Weyl group $W^a$ \cite{Len,OS}.
More precisely, let QBG be the quantum Bruhat graph of $\fg$ \cite{BFP,LSh,LNSSS2}. The set of vertices of QBG is in bijection
with $W$ and the edges are of two sorts: classical edges, coming from the classical Bruhat graph, and
quantum edges, pointing in the opposite direction. A quantum alcove path is an alcove path $p$ projecting to a
path in QBG. A path depends on the starting point $u\in W^a$ and the directions, given by the reduced decomposition
of an element $w$ from the extended affine Weyl group. We denote the set of quantum alcove paths with the data $u$, $w$ by
$\mathcal{QB}(u,w)$. The main combinatorial object of the paper is the generating function
\[
C_u^{w}=\sum_{p \in \mathcal{QB}(u,w)}x^{wt(end(p))}q^{\deg({\rm qwt}(p))}
\]
(see for details Section 1). Let $t_\la$ be the element of the extended affine Weyl group, corresponding to
the weight $\la$. Orr and Shimozono proved that if $\la$ is anti-dominant, then $C_{\rm id}^{t_\la}$ is equal to $E_\la(x,q,0)$; similar
formula exists for the $t=\infty$ specialization as well. We prove the following theorem:

\medskip

\noindent{\bf Theorem B}. Let $\la$ be an anti-dominant weight, $\sigma\in W$. Then
${\rm ch} W_{\sigma(\la)}=C_{\sigma}^{t_{\la}}$.

\medskip

The main tool we use is the recursion relation for the functions $C_u^w$, which we identify with the decomposition procedure for the
generalized Weyl modules.

As a consequence, we develop a new approach to the Chari-Ion theorem \cite{CI}, generalizing the Ion result \cite{I}.
The Ion theorem says that for dual
untwisted Kac-Moody Lie algebras the specialized Macdonald polynomials $E_\la(x,q,0)$ are equal to the character
of the level one Demazure modules. The Chari-Ion theorem claims that one can include the non simply-laced
algebras by replacing the Demazure modules with the Weyl modules: for any dominant weight $\la$ and any simple $\fg$
one has $E_{w_0\la}(x,q,0)={\rm ch} W(\la)$. We note that the proof in \cite{CI} uses the results from \cite{LNSSS3}
(see also \cite{LNSSS4}). In our approach the combinatorics of \cite{LNSSS3} is replaced with the structure
theory of the generalized Weyl modules. More precisely, we show that if one knows the Chari-Ion theorem for fundamental weights
(even a weaker statement, see Remark \ref{fw}), then the theory of the generalized Weyl modules allows to
derive inductively the general $\lambda$ case.

Finally, we use our technique to prove a special case of the Cherednik-Orr conjecture \cite{CO1}, relating the PBW
twisted characters of the Weyl modules to the nonsymmetric Macdonald polynomials at $t=\infty$. We show that
the conjecture holds for the modules $W(m\omega)$, where $\omega$ is a cominuscule fundamental weight.
The $t=0$ and $t=\infty$ specializations for general weights in the quantum settings are studied in \cite{NNS, NS}.

The paper is organized as follows. In Section 1 we recall the formalism of the alcove paths and state the Orr-Shimozono formula for the
nonsymmetric Macdonald polynomials. We then introduce our main combinatorial object -- the function $C_u^w$ -- and derive
recursion relation for it. In Section 2 we introduce the main player from the representation theory side -- the generalized Weyl modules.
We derive the properties of the generalized Weyl modules and describe the connection between the structure of submodules of the generalized
Weyl modules and the alcove paths picture, thus proving Theorem B. Part (i) of Theorem A is a combination of Lemma \ref{lowerbound}
and Theorem \ref{ineq=eq}. Parts (ii) and (iii) of Theorem A are Corollaries \ref{Iongeneralization} and 
\ref{representationinterpretationinfinity} and part (iv) is proved in Theorem  \ref{ineq=eq} (based on the Orr-Shimozono formula \cite{OS}).
In Section 2 we assume that all the claims are true for the rank one and two Lie
algebras. These cases are worked out in Section 3. In Appendix we prove the Cherednik-Orr conjecture for the multiples of 
cominuscule fundamental weights.

\section{Orr-Shimozono formula}
In this section we describe the Orr-Shimozono formula for specializations of nonsymmetric Macdonald polynomials
\cite{OS}.

\subsection{Quantum Bruhat Graph}
Let $\fg$ be a simple Lie algebra of rank $n$ with the root system $\Delta=\Delta_+ \sqcup \Delta_-$. Let
$X$ be the weight lattice of $\fg$ and $W$ be the Weyl group with the set of simple reflections $s_1, \dots, s_n$.
We denote by $\al_i$, $\al_i^\vee$ and $\om_i$, $i=1,\dots,n$ simple roots,
simple coroots and fundamental weights. The positive cone $\bigoplus_{i=1}^n \bZ_{\ge 0}\om_i\subset X$ will be
denoted by $X_+$. For a root $\alpha$ we denote by $s_{\alpha}$ the reflection at this root.
For $w \in W$ let $l(w)$ be the length of the element $w$ in the Bruhat order.

Let $\Delta^\vee$ be the dual root system. The Weyl group of the corresponding Lie algebra $\fg^\vee$ is isomorphic
to $W$. We will use the quantum Bruhat graph (QBG for short) attached to $\Delta^\vee$. The set of vertices of QBG is in one-to-one correspondence
with the Weyl group $W$. The (labeled) edges are of the form $w \stackrel{\al}{\longrightarrow} w s_{\alpha}$, $\al\in\Delta^\vee$;
such an edge shows up in QBG in two possible cases:
\begin{itemize}
\item $l(ws_{\alpha})=l(w)+1$ -- Bruhat edge;
\item $l(ws_{\al})=l(w)-\langle 2\rho,\al \rangle+1$ -- quantum edge.
\end{itemize}
Here $2\rho=\sum_{\gamma\in \Delta_+}\gamma$.

\begin{rem}
In \cite{Lus} Lusztig defined a partial order on the affine Weyl group. This partial order after projection to the
finite Weyl group defines the arrows of the quantum Bruhat graph.
\end{rem}

The following lemma is well known (see e.g. \cite{LNSSS2}).
\begin{lem}\label{inversionarrows}
The longest element  $w_0 \in W$ inverses arrows in the quantum Bruhat graph, i. e. the quantum Bruhat graph contains
an edge $w \stackrel{\alpha}{\longrightarrow} w s_{\alpha}$ if and only if there exists and edge
$w_0ws_{\alpha} \stackrel{\alpha}{\longrightarrow} w_0w$.
\end{lem}
%\begin{proof}
%We have that $l(w_0 w)=l(w_0) - l(w)$, $w_0 w s_\alpha s_\alpha =w_0 w$.
%end{proof}

For example, in types $A$ and $C$ the quantum Bruhat graph can be explicitly described as follows (see  \cite{Len}).
For type $A$ we need the order $\prec_i$ on $1, \dots, n$ starting at $i$,
namely $i \prec_i i + 1 \prec_i\dots  \prec_i n \prec_i 1 \prec_i \dots  \prec_i  i - 1$. It is convenient to think of this order in
terms of the numbers $1,\dots,n$ arranged on a circle clockwise. We make the convention that
whenever we write $a\prec b\prec c\prec \dots$, we refer to the circular order $\prec_a$.

We denote roots in type $A_n$ by $\alpha_{ij}=\al_i+\dots + \al_{j-1}$, $1 \leq i<j \leq n+1$. Recall that the Weyl group of the
Lie algebra of type $A_n$ is isomorphic to the symmetric group $S_{n+1}$.
\begin{prop}\label{descriptiongraphA}(\cite{Len}, Proposition 3.6)
Let $w\in S_{n+1}$ be an element in the Weyl group. Then there exists an edge
$w \stackrel{\alpha_{ij}}{\longrightarrow}  ws_{\alpha_{ij}}$ in the quantum Bruhat graph
if and only if there is no $k$ such that $i<k<j$ and $w(i) \prec w(k) \prec w(j)$.
The edge is quantum if and only if $w(i)>w(j)$.
\end{prop}

In type  $C$ we use the standard ordered alphabet  $1$,$2,\dots$,$n$,$\bar n, \dots $, $\bar 2$, $\bar 1$.
We write the signed permutation from the symplectic Weyl group
as the permutations $\sigma$ of the set $1$,$2,\dots$,$n$,$\bar n, \dots $, $\bar 2$, $\bar 1$ such that $\sigma(\bar i)=\overline {\sigma(i)}$.
We use the standard parametrization of the positive roots in type $C$: $\al_{ij}= \epsilon_i-\epsilon_j$,
$\al_{i\bar j}= \epsilon_i+\epsilon_j$.

\begin{prop}\label{descriptiongraphC}(\cite{Len}, Proposition 5.7)
Let $w$ be an element in the Weyl group of type $C_n$. Then there are edges of three following types:

$1)$ $w \stackrel{\alpha^\vee_{\epsilon_i-\epsilon_j}}{\longrightarrow}  ws_{\alpha^\vee_{ij}}$
if and only if there is no $k$ such that $i<k<j$ and $w(i) \prec w(k) \prec w(j)$;

$2)$ $w \stackrel{\alpha^\vee_{\epsilon_i+\epsilon_j}}{\longrightarrow}  ws_{\alpha^\vee_{i\bar j}}$
if $w(i)>w(\bar j)$ and there is no $k$ such that $i<k<j$ and $w(i) < w(k) < w(j)$;

$3)$ $w \stackrel{\alpha^\vee_{2\epsilon_i}}{\longrightarrow}  ws_{\alpha^\vee_{i\bar i}}$
if and only if there is no $k$ such that $i<k<\bar i$ and $w(i) \prec w(k) \prec w(\bar i)$.

The edge is quantum if and only if $w(i)>w(j)$. In particular there are no quantum edges of type $2)$.
\end{prop}

\subsection{Alcove paths aka LS-galleries}
Let $\widehat{\mathfrak{g}}$
be the non-twisted affine Kac-Moody Lie algebra corresponding to the simple Lie algebra $\fg$.
Let $W^a=\langle s_0, s_1, \dots, s_n \rangle$ be the affine Weyl group
of $\mathfrak{g}^\vee$ ($\mathfrak{g}^\vee$ is the dual Lie algebra with the transposed Cartan matrix).
The finite Weyl group $W$ is generated by the simple reflections $s_1, \dots, s_n$; we denote by $w_0\in W$
the longest element.
Let $Q\subset X$ be the root lattice of $\fg$;
in particular, $W^a$ is isomorphic to the semi-direct product $W\ltimes Q$. We consider the quotient $\Pi=X/Q$.
For example, for $\mathfrak{g}=A_n$
the group $\Pi$ is isomorphic to $\bZ/(n+1)\bZ$.
The extended affine Weyl group $W^e$ is defined as the semi-direct product $W\ltimes X$.
For an element $\la\in X$ we denote by $t_\la$ the corresponding element in $W^e$.
One has $W^e \simeq \Pi \ltimes W^a$.

We consider the $n$-dimensional real vector space $\mathbb{R}\otimes_{\mathbb{Z}}Q$ and the set of hyperplanes (walls)
$H_{\alpha^{\vee}+N\delta}=\lbrace x \in \mathbb{R}\otimes_{\mathbb{Z}}X|\langle \alpha^{\vee},x \rangle=N\rbrace$.
Then alcoves are the connected components of
$\mathbb{R}\otimes_{\mathbb{Z}}X \backslash \cup_{\alpha \in \Delta_+, N \in \mathbb{Z}}H_{\alpha^{\vee}+N\delta}$.
There is a natural action of the affine Weyl group $W^a$ on the set of alcoves (see e.g. \cite{Car,Kac}).
Identifying the alcove $\lbrace a| \langle a,\alpha_i^\vee \rangle>0, i=0, \dots, n \rbrace$ with the
identity element of $W^a$, one obtains a bijection between $W^a$ and the set of alcoves.

Any element of $W^e$ can be written in the form $\pi s_{i_1} \dots s_{i_l}$, $\pi \in \Pi$, $0\le i_l\le {\rm rk}(\fg)$.
In particular, we have such a decomposition for the elements $t_\la$, $\la\in X$. We note also that any element of $W^e$ has the unique
decomposition $w=t_{wt(w)}dir(w)$, where $wt(w)\in X$, $dir(w)\in W$.

Let us consider $|\Pi|$ copies of $\mathbb{R}\otimes_{\mathbb{Z}}Q$ (sheets) indexed by $\Pi$ with the same
action of $W^a$ on all the sheet. The extended affine Weyl group $W^e$ acts on the set of alcoves of all sheets as follows.
For any $\pi \in \Pi$ we identify the alcove
$\lbrace a| \langle a,\alpha_i^\vee \rangle>0, i=0, \dots, n \rbrace$ on the $\pi$-th sheet with
the image (under the action of $\pi$) of the alcove on the initial sheet, corresponding to the identity element of $W^a$.
This  rule defines the action of $W^e$ on the set of alcoves of all sheets (see examples in section 2.3 of \cite{RY}).

For a reduced decomposition $w=\pi s_{i_1} \dots s_{i_l}$ of an element $w\in W^e$ one defines the sequence of
affine real roots:
\begin{equation}\label{beta}
\beta_k(w)=s_{i_l} \dots s_{i_{k+1}} \alpha^\vee_{i_k},\ k=1,\dots,l.
\end{equation}

\begin{rem}
The coroots $\beta_k(w)$ comprise the set of all positive affine coroots which are mapped to the negative roots by $w$.
We note also that $\{\beta_k(w)\}$ is the sequence of labels of walls crossed by a shortest walk from the alcove $w^{-1}$
to the initial alcove of the current sheet (see example on page $6$ in \cite{RY}).
\end{rem}

Let $\bar{b}=(b_1, \dots, b_l)\in \langle 0,1 \rangle^l$ be a binary word and let
$J=\lbrace i| b_i=0 \rbrace$, $J=\lbrace j_1< \dots <j_r \rbrace$.  We call $J$ {\it the set of foldings}.
For an element $u \in W^a$ we set
\[z_0=uw, \ z_{k+1}=z_k s_{\beta_{j_{k+1}}},\ k=0,\dots,r-1.\]

We denote this data by an alcove path $p_J$, so $p_J$ can be written as
\[z_0 \stackrel{\beta_{j_{1}}}{\longrightarrow} z_1 \stackrel{\beta_{j_{2}}}{\longrightarrow}
\dots \stackrel{\beta_{j_{r}}}{\longrightarrow}z_r=:end(p_J).\]
Any alcove path can be projected to the path in a finite group $W$ by the function ${\rm dir}$:
\begin{equation}\label{projectionofpath}
{\rm dir}(z_0) \stackrel{{\rm Re}\beta_{j_{1}}}{\longrightarrow} {\rm dir}(z_1) \stackrel{{\rm Re}\beta_{j_{2}}}{\longrightarrow}
\dots \stackrel{{\rm Re} \beta_{j_{r}}}{\longrightarrow}{\rm dir}(z_r),
\end{equation}
where for an affine root $\beta$ we denote by ${\rm Re}\beta$ the projection to the classical root lattice.
\begin{rem}
All the coroots ${\rm Re}\beta_{j_1},\dots,{\rm Re}\beta_{j_r}$ are negative, see \cite{OS}, Remark 3.17.
In what follows we use both notation $w_1\stackrel{\al}{\longrightarrow} w_2$ and
$w_1\stackrel{-\al}{\longrightarrow} w_2$ to denote the same edge in the quantum Bruhat graph.
\end{rem}
\begin{rem}\label{type}
In what follows we say that any alcove path $p_J$ as above has type $\beta_1,\dots,\beta_l$. We note that
in general the roots  $\beta_1,\dots,\beta_l$ may {\it not} come from a decomposition of $w$.
\end{rem}

\begin{rem}
A path $p_J$ can be also regarded as an LS-gallery \cite{GL} or an alcove walk \cite{RY,OS}. Namely, instead of working with
the Weyl group elements $z_0,z_1,\dots,z_r$ one can think of the chain of alcoves, such that the neighboring alcoves
have a common wall. In this picture the alcoves are assumed to be parametrized by the elements of the extended affine Weyl group.
In more details, for $J=\emptyset$ one considers a path starting at $u$ (on the $\pi$-th sheet) and moving across
the walls according to the reduced decomposition of $w$. Now each element of $J$ produces a fold, meaning that instead of crossing
the corresponding wall, the walk folds (i.e. bounces back). It is important to keep in mind that a path $p_J$ is not
an alcove walk in this sense: in general, the alcoves corresponding to $z_i$ and $z_{i+1}$  do not have a common wall.
\end{rem}

We say that a path $p_J \in \mathcal{QB}(u,w)$ ($p_J$ is a quantum alcove path) if the projection
\eqref{projectionofpath} is a path in the quantum Bruhat graph of $W$.
We say that a path $p_J \in \widetilde{\mathcal{QB}}(u,w)$ if the projection \eqref{projectionofpath} is a path
in the reversed quantum Bruhat graph of $W$. Let $J^-\subset J$ be the set of $j_m\in J$ such that the 
coroot ${\rm Re}(z_m\beta_{j_m})$ is negative.
We note that $j\in J^-$ if and only if the corresponding edge in the quantum Bruhat graph is quantum.

Let $\delta$ be the basic imaginary coroot. For any element of the affine coroot lattice $\mu + N\delta$,
where $\mu\in Q^\vee$ is an element of the root lattice of $\fg^\vee$,  we denote $\deg(\mu + N\delta)=N$.
For an alcove path $p_J$ we define ${\rm qwt}(p_J)=\sum_{j\in J^-} \beta_j$.

\subsection{Generating function}
We are now ready to define the main combinatorial object of the paper.
\begin{definition} \label{Cdef}
For any $u,w \in W^e$ we define:
\[
C_u^{w}(x,q)=\sum_{p_J \in \mathcal{QB}(u,w)}x^{wt(end(p_J))}q^{\deg({\rm qwt}(p_J))}.
\]
\end{definition}

\begin{rem}
It is easy to see that for any $\pi\in\Pi$ one has $C_u^{\pi w}=C_u^w$.
\end{rem}

\begin{rem}
For $\lambda \in -X_+$, $u \in W$ the following equality holds:
\[C_{u}^{t_{w_0(\lambda)}}(x,q)=\left( t^{-l(u)/2}T_u E(x,q,t)\right)|_{t=0},\]
where $T_u$ is the Demazure-Lusztig operator corresponding to $u$ (see \cite{OS}, Corollary 4.4).
\end{rem}

In the rest of this section we describe the properties of the function $C_u^w$.
For a weight $\mu\in X$ recall  the corresponding element $t_\mu\in W^a$.
The following Lemma is obvious.
\begin{lem}\label{shiftC}
For any $\mu \in X$:
\[C_{t_\mu u}^{w}=x^\mu C_u^{w}.\]
\end{lem}
\begin{proof}
There is a bijection between  $\mathcal{QB}(u,w)$ and $\mathcal{QB}(t_\mu u,w)$, sending
$z_i$ to $t_\mu z_i$. Therefore,  for each $p_J\in \mathcal{QB}(u,w)$ passing to
$t_\mu p_J\in \mathcal{QB}(t_\mu u,w)$ means just scaling the corresponding summand in
Definition \ref{Cdef} by $x^\mu$.
\end{proof}

\begin{thm}
\cite{OS}\label{specializationOS0}
Let $\lambda \in X$ be an anti-dominant weight.
%We fix a reduced decomposition $t_\lambda=\pi s_{i_1} \dots s_{i_l}$.
Then
\begin{enumerate}
\item $E_{\lambda}(x; q, 0)=C_{\rm id}^{t_\la}$.
\item $E_{\lambda}(x; q^{-1}, \infty)=\sum_{p_J\in \widetilde{\mathcal{QB}}(\lambda)}x^{wt(end(p_J))}q^{\deg({\rm qwt}(p_J))}$,
\item $E_{\lambda}(x; q^{-1}, \infty)=w_0 C_{w_0}^{s_{i_1} \dots s_{i_l}}.$
\end{enumerate}
\end{thm}
%\begin{proof}
%The second equality follows from Lemma \ref{inversionarrows}.
%\end{proof}

\begin{lem}\label{commutationomegas}
\[t_{\omega_k} s_{\alpha_l + N\delta}t_{-\omega_k}=
                                                  \begin{cases}
                                                    s_{\alpha_l + N\delta}, if  ~ l \neq k \\
                                                    s_{\alpha_l + (N+1)\delta}, if ~ l = k
                                                  \end{cases},
\]
\[t_\lambda(\gamma)=\gamma - \langle {\gamma}^\vee,\lambda \rangle \delta.\]
\end{lem}
\begin{proof}
The first equality is clear and the proof of the second is given in \cite{OS}, formula (2.9).
\end{proof}

Let $\lambda$ be an anti-dominant weight.
Let $t_{-\omega_i}=\pi s_{t_1}\dots s_{t_r}$ be a reduced decomposition and let
$\beta_j^i=\beta_j^i(t_{-\omega_i})$.
Then
$t_{\lambda -\omega_i}=\pi s_{t_1}\dots s_{t_r}t_{\lambda}$.
Let $\beta_1(t_\la),\dots,\beta_a(t_\la)$ be the affine coroots, constructed via the procedure \eqref{beta} for the element $t_\la\in W^e$.

\begin{lem}\label{firstbetas}
The sequence of coroots $\beta_j(t_{\lambda-\omega_i})$ is equal to
\[
\beta_1^i + \langle {\beta_1^i},\lambda \rangle \delta, \dots, \beta_{r_i}^i + \langle {\beta_{r_i}^i},\lambda \rangle \delta,\
\beta_1,\dots,\beta_a.
\]
\end{lem}

\begin{example}\label{funddecompA}
Let us consider a fundamental weight $\omega_i$ for the Lie algebra of type $A_n$. Let
$\pi \in \Pi$ be an element such that $\pi s_i \pi^{-1}=s_{i+1}$
(all indices are modulo $n+1$). Then we have:
\[t_{-\omega_{n+1-i}}=\pi^i (s_{2(n+1-i)}\dots s_{n+2-i})\dots (s_{n-1-i}\dots s_1s_0s_n)(s_{n-i}  \dots s_1s_0).\]
\end{example}

Let $w=t_{-\omega_i}$ and let $r$ be the length of $t_{-\om_i}$. We denote by $\mathcal{QB}(u,\la,\bar\beta)$ all alcove paths of type
$\bar\beta^{i,\lambda}=(\beta^i_1+\langle {\beta^i_1}, \lambda \rangle\delta,
 \dots \beta^i_r+\langle {\beta^i_r}, \lambda \rangle\delta)$
starting at $ut_{\la-\omega_i}$  (see Remark \ref{type}). In the next theorem for an anti-dominant weight $\la$ we express
the generating function $C_u^{t_{\lambda-\omega_i}}$ in terms of the functions $C^{t_{\lambda}}_{\kappa}$ for certain Weyl group
elements $\kappa$, thus getting a kind of induction (on $\lambda$)..  
\begin{thm}\label{combinatorialdecomposition}
Let $\lambda\in -X_+$. Then for $u\in W^a$ the following holds:
\[C_u^{t_{\lambda-\omega_i}}=\sum_{p \in \mathcal{QB}(u,\lambda,\bar\beta^{i,\lambda})}q^{{\rm deg}({\rm qwt}(p))}
C^{t_{\lambda}}_{end(p)t_{-\lambda}}.\]
Further, if $u \in W$, then
\[C_u^{t_{\lambda-\omega_i}}=\sum_{p \in \mathcal{QB}(u,\lambda,\bar\beta^{i,\lambda})}q^{{\rm deg}({\rm qwt}(p))}
C^{t_{\lambda}}_{dir(end(p))}x^{wt(end(p))-dir(end(p))\lambda}.
\]
\end{thm}
\begin{proof}
Recall the definition of $C_{u}^w$:
\[
C_u^{w}=\sum_{p_J \in \mathcal{QB}(u,w)}x^{wt(end(p_J))}q^{\deg({\rm qwt}(p_J))}.
\]
%Let $r$ be the length of $t_{-\omega_i}$.
An alcove path $p_J\in \mathcal{QB}(u,w)$ is determined by the
sequence of affine coroots $\beta_1,\dots,\beta_r, \beta_{r+1},\dots,\beta_{a+r}$ (for some nonnegative integer $a$)
and a binary word $\{b_1,\dots,b_{a+r}\}$. Now given an alcove path $p_J\in \mathcal{QB}(u,w)$ we divide it into
two parts: the first part $p$ is determined by the data
\[
\beta_1,\dots,\beta_r \text{ and } \{b_1,\dots,b_{r}\}
\]
and the second part $p'$ is defined by the remaining part of the data for $p_J$.
Then $p$ belongs to $\mathcal{QB}(u,\la,\bar\beta^{i,\lambda})$ (see Lemma \ref{firstbetas})
and $p'$ belongs to $\mathcal{QB}(end(p)t_{-\la},t_\la)$.
Moreover, the contribution of $p$ is exactly $q^{{\rm deg}({\rm qwt}(p))}$ and the terms corresponding to $p'$
sum up to $C^{t_{\lambda}}_{end(p)t_{-\lambda}}$. Finally, the second part of the Theorem follows from Lemma \ref{shiftC}.
\end{proof}

\subsection{Combinatorics of coroots}
Recall that for an affine coroot $\beta$ we write $\beta={\rm Re}(\beta) +\deg (\beta)\delta$.

\begin{prop}\label{descriptionbeta}
$a).$\ For any reduced decomposition of $t_{-\omega_i}$ the coroots $\beta_j(t_{-\om_i})$ satisfy the following properties:
\begin{itemize}
\item $\lbrace {\rm Re} \beta_j^i \rbrace=\lbrace \gamma \in \Delta^\vee_-|\langle \gamma, \omega_i \rangle<0 \rbrace$,
\item $|\lbrace j|{\rm Re}\beta_j^i=\gamma \rbrace|=-\langle \gamma, \omega_i \rangle$,
\item For any $\gamma$ the set $\lbrace \beta_j|{\rm Re}\beta_j^i=\gamma \rbrace$ is equal to
$\lbrace \gamma+\delta,  \dots, \gamma - \langle \gamma, \omega_i \rangle \delta\rbrace$.
\end{itemize}

$b).$\ There exists a reduced decomposition of $t_{-\omega_i}$ giving the following order on $\beta$'s.
We set $i_1=i$, and let $i_k$, $k=2, \dots, n$, be some ordering of the set $\lbrace 1,\dots, n \rbrace \backslash \lbrace i \rbrace$.
Let us write $\beta_j^i=-a_{i_1}\alpha_{i_1}^{\vee} - \dots - a_{i_n}\alpha_{i_n}^{\vee}+ D\delta$.
Then the order on $\beta$'s is given
by the lexicographic order on the vectors $(\frac{a_{i_1}}{D},\frac{a_{i_2}}{a_{i_1}}, \dots, \frac{a_{i_n}}{a_{i_1}})$.
\end{prop}
\begin{proof}
For the Lie algebras of type $A$ our proposition can be derived from Example \ref{funddecompA} by the direct computation.
In general, for $\gamma\in\Delta^\vee_-$ the number $-\langle \gamma, \omega_i \rangle$ is equal to the number of walls 
with labels $-\gamma+\bZ\delta$ between
the alcove ${\rm id}$ and the alcove $\pi^{-1}t_{\omega_i}$, where $\pi\in\Pi$ is fixed by the condition that $\pi^{-1}t_{\omega_i}$
belongs to the zeroth sheet.
In other words, walking from the initial alcove to the alcove corresponding to the element $\pi^{-1}t_{\omega_i}$ we need
to cross $-\langle \gamma, \omega_i \rangle$ walls with labels $-\gamma+\bZ\delta$. It is easy to see that these walls are
$\gamma + \delta, \dots, \gamma +\langle \gamma, \omega_i \rangle \delta$.
This proves the statement about the set $\lbrace {\rm Re} \beta_j^i \rbrace$.

Now our goal is to prove the existence of a reduced decomposition of $t_{-\omega_i}$ such that the properties from part $b)$
of our Proposition hold. This is equivalent to
finding an alcove walk from the identity alcove to the alcove, corresponding  to $\pi^{-1}t_{\omega_i}$, of the minimal possible length.

We order the elements of the set $\{1,\dots,n\}$ in the following way. Put $i_1=i$, and let $i_k$, $k=2, \dots, n$ be any ordering
of the set $\lbrace 1,\dots, n \rbrace \backslash \lbrace i \rbrace$.
We take some set of positive real numbers $\epsilon_k$, $k=2, \dots, n$ such that $\epsilon_2 <<1$, $\epsilon_{k+1}<<\epsilon_k$.
Let us consider the segment from the point $\sum_{k=2}^n \epsilon_{k} \omega_{i_k}$ to the point
$\omega_i +\sum_{k=2}^n \epsilon_{k} \omega_{i_k}$.
We write the set of walls crossed by this segment (see picture \eqref{ApC2} for the example in type $C_2$).
We consider a point $p=s\omega_i +\sum_{k=2}^n \epsilon_{i} \omega_{i_k}$, $0\le s\le 1$
of this segment and an arbitrary coroot $\gamma=-(a_{1}\alpha_i^\vee+a_{2}\alpha_{i_2}^\vee+\dots +a_n\alpha_{i_n}^\vee)$.
The condition $p\in H_{\gamma+D\delta}$, $D\in\bZ$ (i.e. $p$ belongs to some wall)  reads as
\[
\langle p,a_{1}\alpha_i^\vee+a_{2}\alpha_{i_2}^\vee+\dots +a_n\alpha_{i_n}^\vee\rangle =
s a_1 +\sum_{k=2}^n \epsilon_{k} a_k=D.
\]
Therefore for $\epsilon_k$ small enough the coroot $\beta^i_j$ with smaller $a_1/D$ comes earlier and $D/a_1\le 1$.
Now assume that the ratio $a_1/D$ is fixed. Then $s=D/a_1-\xi$, where
\[
\xi=\sum_{k=2}^n \epsilon_{k} \frac{a_k}{a_1}.
\]
Hence, the smaller is $\xi$ the larger is $s$ and thus the root with smaller $a_2/a_1$ comes earlier (recall
$\epsilon_2>>\epsilon_3>>\dots$). We proceed with $a_3/a_1$, etc.
\end{proof}
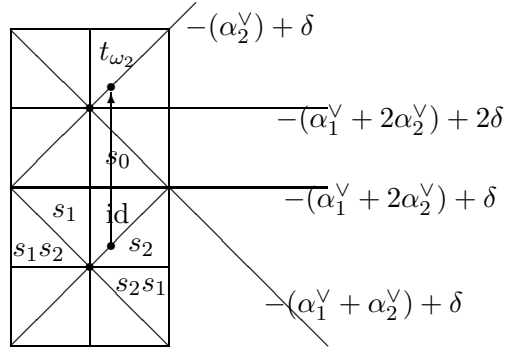
\begin{figure}
\begin{picture}
(140.00,140.00)(-5,00)
 \put(40.00,40.00){\circle*{3.00}}
 \put(40.00,100.00){\circle*{3.00}}
 \put(48.00,48.00){\circle*{3.00}}
 \put(48.00,108.00){\circle*{3.00}}
 \put(48.00,48.00){\vector(0,1){58}}
 \put(153.00,65.00){\makebox(1,1){$-(\alpha_1^\vee+2\alpha_2^\vee)+\delta$}}
 \put(153.00,95.00){\makebox(1,1){$-(\alpha_1^\vee+2\alpha_2^\vee)+2\delta$}}
 \put(100.00,130.00){\makebox(1,1){$-(\alpha_2^\vee)+\delta$}}
 \put(143.00,25.00){\makebox(1,1){$-(\alpha_1^\vee+\alpha_2^\vee)+\delta$}}
 \put(50.00,60.00){\makebox(1,1){\rm id}}
 \put(30.00,60.00){\makebox(1,1){$s_1$}}
 \put(20.00,45.00){\makebox(1,1){$s_1s_2$}}
 \put(59.00,47.00){\makebox(1,1){$s_2$}}
 \put(59.00,32.00){\makebox(1,1){$s_2s_1$}}
 \put(50.00,80.00){\makebox(1,1){$s_0$}}
 \put(50.00,120.00){\makebox(1,1){$t_{\omega_2}$}}
 \put(10.00,10.00){\line(1,0){60}}
 \put(10.00,40.00){\line(1,0){60}}
 \put(10.00,70.00){\line(1,0){120}}
 \put(10.00,100.00){\line(1,0){120}}
 \put(10.00,130.00){\line(1,0){60}}
 \put(10.00,10.00){\line(0,1){120}}
 \put(40.00,10.00){\line(0,1){120}}
 \put(70.00,10.00){\line(0,1){120}}
 \put(10.00,10.00){\line(1,1){60}}
 \put(10.00,70.00){\line(1,1){70}}
 \put(70.00,10.00){\line(-1,1){60}}
 \put(130.00,10.00){\line(-1,1){120}}
\end{picture}
\caption{Alcove walk in type $C_2$}\label{ApC2}
\end{figure}

\begin{cor}\label{order}
$i)$ $\beta_1^i=-\alpha^\vee_i+\delta$,\\
$ii)$ if $\gamma=\tau + \eta$, $\tau, \eta \in \Delta_+^{\vee}$, ${\rm Re}\beta_j^i=-\gamma$, then
\begin{multline*}
|\lbrace k|{\rm Re} \beta_k^i =-\gamma, k \leq j\rbrace|=\\
|\lbrace k|{\rm Re} \beta_k^i =-\tau, k \leq j\rbrace|+
|\lbrace k|{\rm Re} \beta_k^i =-\eta, k \leq j\rbrace|.
\end{multline*}
$iii)$ Let $\tau, \eta\in \Delta_+^\vee$ be roots such that $\tau+2\eta\in \Delta_+^\vee$. Consider a
subsequence $\beta^i_{j_k}, k=1,\dots,p$ consisting of all roots with the property
$-{\rm Re}\beta^i_{j_k}\in \lbrace \tau,\eta, \tau+\eta, \tau+2\eta \rbrace$ ($j_k<j_{k+1}$). Then the subsequence
$-{\rm Re}\beta^i_{j_k}, k=1,\dots,p$ is a concatenation of copies of two following sequences:
\begin{equation}\label{sequencecC_2}
\eta, \tau+2\eta, \tau+\eta,\tau+2\eta ~ {\rm and}~
\tau,\tau+\eta,\tau+2\eta.
\end{equation}
\end{cor}
\begin{proof}
The first statement is obvious.
To prove the second, let $\tau=a_{1}\alpha_i^\vee+a_{2}\alpha_{i_2}^\vee+\dots+a_{n}\alpha_{i_n}^\vee$,
$\eta=b_{1}\alpha_i^\vee+b_{2}\alpha_{i_2}^\vee+\dots+b_{n}\alpha_{i_n}^\vee$.
Assume that ${\beta_j^i}=-\eta-\tau+(a_1+b_1-r)\delta$.
Then we have
\[
|\{\beta^i_j:\ j\le m, -{\rm Re}\beta_j^i=\tau+\eta\}|=r+1
\]
(see Proposition \ref{descriptionbeta}).
We count a number of ${\beta_m^i}$, $m<j$, such that
${\rm Re}{\beta_m^i}=-\eta$ or ${\rm Re}{\beta_m^i}=-\tau$. Note that if for a number $o_1$ we have the
inequality
\begin{equation}\label{o1}
\frac{a_1}{a_1-o_1}<\frac{a_1+b_1}{a_1+b_1-r},
\end{equation}
then
$\tau+(a_1-o_1)\delta=\beta_m^i$ for some $m<j$; if
\begin{equation}\label{o2}
\frac{b_1}{b_1-o_2}<\frac{a_1+b_1}{a_1+b_1-r},
\end{equation}
then
$\eta+(b_1-o_2)\delta=\beta_m^i$ for some $m<j$. We also note that each of the converse inequalities implies
the absence of the  $\beta_m^i$ with the real part equal to $\tau$ or $\eta$.
We rewrite inequalities \eqref{o1} and \eqref{o2} in the form $o_1 < \frac{a_1r}{a_1+b_1}$, $o_2 < \frac{b_1r}{a_1+b_1}$. Note that if
$\frac{a_1r}{a_1+b_1}$ does not belong to $\mathbb{Z}$, then the number of solutions of these inequalities is equal to $r+1$
and the claim $ii)$ is proved. If the number $\frac{a_1r}{a_1+b_1}$ is integer then the number of solutions is equal to $r$.
In this case consider $o_1=\frac{a_1r}{a_1+b_1}$, $o_2=r-o_1$. Then we have
$\frac{a_1}{a_1-o_1}=\frac{b_1}{b_1-o_2}=\frac{a_1+b_1}{a_1+b_1-r}$ and using lexicographic order
we have that $-\tau+(a_1-o_1)\delta=\beta_{m_1}^i$,  $-\eta+(b_1-o_2)\delta=\beta_{m_2}^i$ and exactly one of the
numbers $m_1, m_2$ is less then $j$. This completes the proof of $ii)$.

Now let us prove $iii)$. We still use the notation of the previous proof. The claim is the easy consequence of $ii)$ and
the lexicographic order if $a_1=0$ or $b_1=0$ (in this case there is only one type of sequences \eqref{sequencecC_2}).

Note that the situation of $iii)$ is impossible for a simply-laced $\mathfrak g$. For $\mathfrak g \simeq B_n, C_n$ we have 
$a_1+2b_1\leq 2$,
so this case is already proven. Case $\mathfrak g \simeq G_2$ will be considered in \eqref{betasG21},\eqref{betasG22}.
If $\mathfrak g \simeq F_4$ then the direct observation of the root system says that the only possibility
of such $\eta$, $\tau$ with $a_1 \neq 0$, $b_1 \neq 0$ is $i=2$,
$\tau=2\alpha_1^\vee+2\alpha_2^\vee+2\alpha_3^\vee+\alpha_4^\vee$, $\eta=\alpha_2^\vee$. In this case the claim can be proven
by an easy direct computation.

\end{proof}

\begin{example}
Let $\fg$ be of type $A_n$. Then the set $\beta_j(t_{-\om_i})$ is equal to $\lbrace \beta^i_k\rbrace=\lbrace
-\alpha_u - \dots - \alpha_v + \delta\rbrace$, $u \leq i \leq v$ in some lexicographic order. Note that in this case
if for some positive root $\gamma$: ${\rm Re}\beta^i_k={\rm Re}\beta^i_s - \gamma$, then $r>s$.
\end{example}

\section{Generalized Weyl Modules}
\subsection{Definition and basic properties}
Let $\mathfrak{g}=\mathfrak{n}_-\oplus \mathfrak{h} \oplus \mathfrak{n}_+$ be the Cartan decomposition of $\fg$.
For a positive root $\al$ let $e_\al\in\fn_+$ and $f_{-\al}\in\fn_-$ be the Chevalley generators.
The weight lattice $X$ contains the positive part $X_+$, containing all fundamental weights.
For $\la\in X_+$ we denote by $V_\la$ the irreducible
highest weight $\fg$-module with highest weight $\la$.

Let $\widehat{\fg}=\fg\T\bK[t,t^{-1}]\oplus\bK c\oplus\bK d$ be the corresponding affine Kac-Moody Lie algebra,
where $c$ is central and $d$ is the scaling element. Recall  the basic imaginary root $\delta\in(\fh^{af})^*$, where
$\fh^{af}=\fh\T 1\oplus\bK c\oplus\bK d$.
The Lie algebra $\widehat\fg$ has the Cartan decomposition
$\widehat\fg=\fn^{af}\oplus\fh^{af}\oplus\fn_-^{af}$; in particular, $\fn^{af}=\mathfrak{n}_+\T 1\oplus \fg\T t\bK[t].$
For $x\in\fn_+$ we sometimes denote the element $x\T 1\in\fn^{af}$ simply by $x$.
%\begin{rem}
%The definition of $\fn^{af}$ differs from the standard definition of the positive part of the affine Kac-Moody Lie algebra.
%Namely, the positive part is usually defined as the span of the root vectors
%of $\widehat\fg$, corresponding to positive (affine) roots. However, in our definition $\fn^{af}$ contains all the vectors
%of the form $f_\al\T 1$ (but not $e_\al\T 1$). The reason we change the notation is that we are going to define
%cyclic representations $W_\mu$ of $\fn^{af}$ (labeled by a weight $\mu$ of $\fg$), generalizing the classical
%Weyl modules. In particular, for dominant weights $\mu$ the $W_\mu$ is isomorphic to the classical Weyl module $W(\mu)$.
%However, the cyclic vector of $W(\mu)$ is killed by $\fn_+$, but not by $\fn_-$. This change of the standard notation
%is not too painful, since  the elements $e_\al$ and $f_\al$ can be interchanged by the Chevalley involution.
%\end{rem}

\begin{definition}\label{weylmodules}
Let $\mu=\sigma(\lambda)$, $\sigma\in W$, $\lambda\in X_-$. Then the generalized Weyl module
$W_{\mu}$ is the cyclic $\fn^{af}$ module with a generator $v$ and the following relations:
\begin{gather}
\label{weylvanishing0}
h\T t^k v=0 \text{ for all } h\in\fh, k>0;\\
\label{weylvanishing1}
(f_{\alpha}\otimes t) v=0, \ \alpha \in \sigma(\Delta_-)\cap \Delta_-;\\
\label{weylvanishing2}
(e_{\alpha}\otimes 1) v=0, \ \alpha \in \sigma(\Delta_-)\cap\Delta_+;\\
\label{weylbound1}
(f_{\sigma(\alpha)}\otimes t)^{-\langle \alpha^\vee, \lambda \rangle+1} v=0, \ \alpha \in \Delta_+, \sigma(\alpha) \in \Delta_-;\\
\label{weylbound2}
(e_{\sigma(\alpha)}\otimes 1)^{-\langle \alpha^\vee, \lambda \rangle+1} v=0,\ \alpha \in \Delta_+,\ \sigma(\alpha) \in \Delta_+.
\end{gather}
\end{definition}

In what follows we use the following notation.
For an element $\sigma \in W$ and $\alpha \in \Delta_+$ we set
\begin{gather*}
\widehat \sigma(f_{-\alpha}\T t)=
                            \begin{cases}
                             f_{-\sigma(\alpha)}\T t, ~ if ~ \sigma(\alpha) \in \Delta_+ \\
                             e_{-\sigma(\alpha)}\otimes 1, ~ if ~ \sigma(\alpha) \in \Delta_- \\
                            \end{cases},\\
\widehat \sigma(e_{\alpha}\T 1)=
                            \begin{cases}
                             e_{\sigma(\alpha)}\T 1, ~ if ~ \sigma(\alpha) \in \Delta_+ \\
                             f_{\sigma(\alpha)}\otimes t, ~ if ~ \sigma(\alpha) \in \Delta_- \\
                            \end{cases}.
\end{gather*}														
We also define the action of $\widehat\sigma$ on roots as follows:
\begin{gather*}
\widehat \sigma(-\alpha+\delta)=   \begin{cases}
                             {-\sigma(\alpha)+\delta}, ~ if ~ \sigma(\alpha) \in \Delta_+ \\
                             {-\sigma(\alpha)}, ~ if ~ \sigma(\alpha) \in \Delta_- \\
                            \end{cases},\\
\widehat \sigma(\alpha)=   \begin{cases}
                             {\sigma(\alpha)}, ~ if ~ \sigma(\alpha) \in \Delta_+ \\
                             {\sigma(\alpha)}+ \delta, ~ if ~ \sigma(\alpha) \in \Delta_- \\
                            \end{cases}.										
\end{gather*}

In the following lemma we prove that the generalized Weyl modules are well defined, i.e. $W_\mu$ does
not depend on the choice of $\sigma$ and $\la$ (such that  $\sigma(\lambda)=\mu$).
\begin{lem}\label{welldefined}
The modules $W_\mu$ are well defined.
\end{lem}
\begin{proof}
We first note that for $\la_1,\la_2\in X_-$, the equality $\sigma_1(\la_1)=\sigma_2(\la_2)$ implies $\la_1=\la_2$.
So let us fix $\la\in X_-$, $\sigma\in W$ and $\kappa\in {\rm stab}(\la)\subset W$. Our goal is to show
that the sets of relations \eqref{weylvanishing1}, \eqref{weylvanishing2}, \eqref{weylbound1}, \eqref{weylbound2}
coincide for the pairs $\sigma$,$\lambda$ and $\sigma\kappa$,$\lambda$.
Note that for any $\eta \in \Delta$: $\langle (\kappa^{-1}\eta)^\vee,\lambda \rangle=\langle \eta^\vee,\kappa\lambda \rangle=
\langle \eta^\vee,\alpha \rangle$.
Assume that for some $\gamma \in \Delta_+$ $\kappa(\gamma)\in \Delta_-$. Then we have:
\[0 \leq \langle \gamma^\vee,\lambda \rangle=\langle (\kappa\gamma)^\vee,\lambda \rangle\leq 0.\]
Therefore $\langle \gamma^\vee,\lambda \rangle=0$ and
$\widehat\sigma(e_{\gamma})v=\widehat{\sigma\kappa}(e_{\gamma})v=0$ in both modules.

Now assume that $\gamma \in \Delta_-$, $\kappa(\gamma)\in \Delta_-$. Then we have the relation in $W_{\sigma\kappa}$:
\[\widehat{\sigma\kappa}e_{\kappa^{-1}\gamma}^{-\langle (\kappa^{-1}\eta)^\vee,\lambda \rangle+1}v=0.\]
Thus using the relation $\widehat{\sigma\kappa}e_{\kappa^{-1}\gamma}=\widehat\sigma e_{\gamma}$ we obtain all needed relations.
\end{proof}

\begin{rem}
The algebra $\fn^{af}$ does not contain the finite Cartan subalgebra $\fh$. However, sometimes
it is convenient to have extra operators from $\fh$ acting on $W_\mu$ (see the definition below).
The reason we do not want
to extend $\fn^{af}$ to the affine Borel subalgebra is that the structure and properties of the module $W_\mu$
do not depend on the weight defining the $\fh$-action on the cyclic vector.
\end{rem}

\begin{definition}
For $\nu\in X$ we define $W_{\mu}^\nu$ to be the $\fn^{af}\oplus\fh$-module defined by the relations
\eqref{weylvanishing0}-- \eqref{weylbound2}
plus additional relations $hv=\nu(h)v$ for all $h\in\fh$.
If $\nu=\mu$, we omit the upper index and write $W_{\mu}$ for $W_{\mu}^\mu$.
\end{definition}

We note that all the modules $W_{\mu}^\nu$ with fixed $\mu$ are isomorphic after restriction to $\fn^{af}$.
The modules $W_\mu^\nu$ are naturally graded by the Cartan subalgebra $\fh$. They also carry additional
degree grading defined by two conditions: ${\rm deg} (v)=0$ and the operators $x_\gamma\T t^k$ increase
the degree by $k$. We define the character by the formula:
\[
{\rm ch} W_\mu^\nu=\sum \dim W_\mu^\nu[\gamma,k] x^\gamma q^k ,
\]
where $W_\mu^\nu[\gamma,k]$ consists of degree $k$ vectors of $\gh$-weight $\gamma$.
In particular, we write ${\rm ch} W_\mu$ for the character of $W_\mu^\mu$.
\begin{rem}
The character ${\rm ch}W_\mu^\nu$ is a Laurent polynomial in $x^{\om_i}$ and $q$. Substituting
$x_{\om_i}=1, q=1$, one gets  ${\rm ch}W_\mu^\nu(1,1)=\dim W_\mu^\nu$.
\end{rem}

\begin{rem}
The generalized Weyl modules are not isomorphic in general to the Demazure modules (note that both
are representations of $\fn^{af}$). Namely, the defining relations for the Demazure modules \cite{J,FL2,N}
are of the form $(e_\al\T t^s)^mv=0$, $s\ge 0$ and
$(f_\al\T t^s)^mv=0$, $s> 0$ for $m$ large enough. We note that the conditions are given for all possible $s$.
For the generalized Weyl modules the set of relations is much smaller: one only requires
$e_\al\T 1$ and $f_\al\T t$ to vanish being applied large enough number of times.
For example, if $\fg=\msl_3$ and $\mu=\om_1+\om_2$, then $W_\mu$ is not isomorphic to a Demazure module.
\end{rem}

The classical definition of Weyl modules $W(\la)$, $\la\in X_+$ (\cite{CP,CL,FL1,FL2}) is slightly different from the definition of
$W_{\mu}$. Namely, $W(\la)$ is a cyclic $\mathfrak{g}\T\bK[t]$ module with generator $w$ subject to the following defining
relations:
\begin{gather}
\label{classicalweyl1}
h \otimes t^k w=0, k \geq 1;\ h\T 1 w=\la(h)w\text{ for all } h\in\fh;\\
\label{classicalweyl2}
e_{\alpha}\otimes t^k w=0, k \geq 0;\ (f_{-\alpha} \otimes 1)^{\langle \alpha^\vee, \lambda \rangle+1}w=0,
\text{ for all } \al\in\Delta_+.
\end{gather}

\begin{lem}
For an anti-dominant weight $\lambda$ one has the isomorphism of $\fn^{af}$ modules $W(w_0\lambda)\simeq W_{\lambda}$.
\end{lem}
\begin{proof}
Let us consider the module $W_{\la}^{\la}$ (i.e. we define the $\fh$ action on $W_{\la}$ by the relation
$h\T 1 v=\la(h)v$).
By the BGG resolution, the subspace ${\rm U}(\fn_+)v\subset W_\la$ is isomorphic to $V_{w_0\la}$
($v$ is identified with the lowest weight vector of $V_\la$) and we can extend the structure
of $\fn^{af}\oplus \fh$ module
on $W_{\la}$ to the structure of $\fg\T\bK[t]$ module, saying that $f_{\alpha}\otimes t^k v=0$.
Now the extended module is defined by the $w_0$-twisted relations \eqref{classicalweyl1},\eqref{classicalweyl2} and
hence is isomorphic to $W(w_0\la)$.
\end{proof}

It is well known that the level zero subspace of the classical Weyl module $W(\la)$ is isomorphic to the irreducible
$\fg$-module $V_\la$.
Here is the analogue for $W_{\mu}$, $\mu=\sigma\la$, $\la\in X_-$ (the vector $v_{\mu}$ below is the weight
$\mu$ extremal vector in $V_\la$).
\begin{lem}\label{oppDem}
The subspace ${\rm U}(\fn_+)v\subset W_{\mu}$ is isomorphic to the Demazure module
${\rm U}(\fn_+)v_{\mu}\subset V_{w_0\la}$.
\end{lem}
\begin{proof}
The subspace  ${\rm U}(\fn_+)v\subset W_{\mu}$ is defined as the cyclic $\fn_+$ module with the defining relations
$e_\al v=0$, if $\al>0$, $\sigma(\al)<0$ and $e_{\sigma(\al)}^{\bra\al^\vee,\la\ket+1} v=0$, if $\al>0$, $\sigma(\al)>0$.
These are exactly the defining relations for the Demazure module.
\end{proof}

Now we need one more definition of the module depending on an arbitrary element of the weight lattice.
Let $V$ be a $\mathfrak{g\otimes \mathbb{K}[t]}$-module. Then for any constant $z\in\bK$ it has the following natural structure of
$\fn^{af}$-module: for $x \in \mathfrak{g}$, $v \in V$
\[(x \otimes t^i) v= x\otimes (t-z)^i v.\]
We denote such a module by $V^z$.

Let $\mu=\sigma(\lambda)$, where $\sigma \in W$, $\lambda\in X_-$ and let
$w_0\lambda=\sum_{j=1}^{M} \omega_{k_j}$,  $1\le k_j\le n$ are arbitrary (possibly, coinciding) numbers.
We consider a vector $\bar z=(z_1, \dots, z_M) \in \mathbb{K}^M$, where $z_a \neq z_b$ if $a \neq b$.
Let $W(\omega_{k_j})$, $j=1, \dots, M$ be the Weyl modules ($\fg\otimes \mathbb{K}[t]$ modules), corresponding to
fundamental weights with cyclic lowest weight vectors $w_j\in W(\omega_{k_j})$.
There is a structure of a cyclic $\fn^{af}$-module on the tensor product $\bigotimes_{i=1}^M W^{z_j}(\omega_{k_j})$
with the cyclic vector $\sigma(w_1 \otimes \dots \otimes w_M)$ given by construction of the
fusion product (see \cite{FeLo},\cite{FL2}). Namely, let ${\rm U}(\fn^{af})_s$ be the grading on the universal enveloping algebra
such that $x\T t^s\in {\rm U}(\fn^{af})_s$, $x\in\fg$. Then one can induce a filtration $F_s$ on  $\bigotimes_{j=1}^M W^{z_j}(\omega_{k_j})$
by the formula
\[
F_s={\rm U}(\fn^{af})_s\sigma(w_1 \otimes \dots \otimes w_M).
\]

\begin{definition}\label{fusionmodules}
The $\fn^{af}$ module $W(\omega_{k_1})_{\sigma}* \dots * W(\omega_{k_M})_{\sigma}$ is the associated graded module
$\bigoplus_{s\ge 0} F_s/F_{s-1}$.
\end{definition}

\begin{example}
The definition above works for arbitrary $\fg\T\bK[t]$ modules, not necessarily for fundamental Weyl modules.
For example, let us take irreducible highest weight $\fg$ module $V_{w_0}\la$ with lowest weight vector $v$
and let us make $V_{w_0\la}$ into $\fg\T\bK[t]$ module
saying that $x\T t^k$ acts trivially unless $k=0$. Obviously, the operators $e_\al$ and $f_{-\al}$ generate the whole
space $V_\la$ from the vector $\sigma(v)$. Now we attach degree one to all the operators $f_{-\al}$ and degree zero
to all the operators $e_{\al}$.  Then one has an increasing filtration $F_s$ on $V_\la$, where $s$ is the degree of a monomial
applied to $\sigma(v)$. The associated graded space is a module over $\fn^{af}$, constructed by the procedure
in Definition \ref{fusionmodules} for $M=1$.
\end{example}

\begin{lem}  \label{lowerbound}
Let $w_0\la=\sum_{j=1}^M \omega_{k_j}$. Then there is a surjective homomorphism of $\fn^{af}$-modules
\[
W_{\sigma(\la)} \twoheadrightarrow W(\omega_{k_1})_\sigma * \dots * W(\omega_{k_M})_{\sigma}.
\]
In particular
$\dim W_{\sigma(\lambda)} \geq \prod_{j=1}^M \dim W(\omega_{k_j}).$
\end{lem}
\begin{proof}
It is easy to check that relations from Definition \ref{weylmodules} hold in $W(\omega_{i_1})_\sigma * \dots * W(\omega_{i_M})_{\sigma}$.
\end{proof}

It has been proven in \cite{FL2} that the map from Lemma \ref{lowerbound} is an isomorphism for $\sigma={\rm id}$ for Lie algebras of types
A, D, E. In particular, $\dim W_{\sigma(\la)} = \prod_{j=1}^M \dim W(\omega_{k_j}).$

\subsection{QBG and Weyl modules}
In the following lemma we give a criterion of the existence of edges in the quantum Bruhat graph.
In part $ii)$ by a short root we mean a root such that there exists another root of a larger length.
For example, for simply laced algebras we have no short roots.

\begin{lem}\label{edgesinQBG}
For $\sigma \in W$, $\gamma \in \Delta^\vee_+$ the two following statements are equivalent:

$i)$ there is an edge in the quantum Bruhat graph $\sigma \stackrel{\gamma}{\longrightarrow} \sigma s_{\gamma}$;

$ii)$ there are no elements $\alpha, \beta \in \Delta^\vee_+$ such that
$\alpha, \beta \neq \gamma$, $\alpha + \beta=2\frac{\langle \alpha, \gamma \rangle}{\langle \gamma, \gamma \rangle}\gamma$,
$\widehat{\sigma}(\alpha) + \widehat{\sigma}(\beta)=
2\frac{\langle \alpha, \gamma \rangle}{\langle \gamma, \gamma \rangle}\widehat{\sigma}(\gamma)$;
if $\sigma\gamma\in\Delta^\vee_-$, then additionally
$\gamma$ is not a short nonsimple root contained in a rank two subalgebra, generated by roots from $\Delta^\vee_+$.
%Equivalently the subalgebra spanned on roots in $\mathbb{Q}\langle \widehat{w}(\alpha),
%\widehat{w}(\beta),\widehat{w}(\gamma) \rangle$ is of rank $3$ or this algebra is nonsimply-laced and $\gamma$ is short.
\end{lem}

\begin{proof}
Assume that $\sigma(\gamma) \in \Delta^\vee_+$, then $\sigma s_\gamma(\gamma)\in \Delta^\vee_-$.
We note that $\sigma s_{\gamma}>\sigma$ in the Bruhat order and $l(\sigma)$ is
equal to $|\{\eta \in \Delta^\vee_+| \sigma(\eta) \in \Delta^\vee_-\}|$. Consider the set $\Delta^\vee_+\cap s_{\gamma}\Delta^\vee_+$.
Obviously the numbers of elements of this set
sent to $\Delta^\vee_-$ by $\sigma$ and $\sigma s_{\gamma}$ are equal. Now consider the set $\Delta^\vee_+\cap s_{\gamma}\Delta^\vee_-$.
If $\alpha \in \Delta^\vee_+, s_\gamma(\alpha) \in \Delta^\vee_-$, then $\langle \alpha, \gamma \rangle>0$.
If $\sigma(\alpha) \in \Delta^\vee_-$ then $\sigma s_{\gamma}(\alpha)=\sigma(\alpha)-
2\frac{\langle \alpha, \gamma \rangle}{\langle \gamma, \gamma \rangle}\sigma(\gamma)\in \Delta^\vee_-$. 
Hence $l(\sigma s_{\gamma}) \geq l(\sigma)+1$.
Assume that $l(\sigma s_{\gamma}) \geq l(\sigma)+1$. Then there exists such
$\sigma(\alpha) \in \Delta^\vee_+$, $s_\gamma(\alpha)\in \Delta^\vee_- $, $\sigma s_{\gamma}(\alpha) \in \Delta^\vee_-$.
Thus there exist $\alpha, \beta \in \Delta^\vee_+$
such that $\alpha + \beta=2\frac{\langle \alpha, \gamma \rangle}{\langle \gamma, \gamma \rangle}\gamma$,
$\widehat{\sigma}(\alpha) + \widehat{\sigma}(\beta)=
2\frac{\langle \alpha, \gamma \rangle}{\langle \gamma, \gamma \rangle}\widehat{\sigma}(\gamma)$.
Converse statement can be proven in the same way.

Assume that $\sigma(\gamma) \in \Delta^\vee_-$. Then $\sigma s_{\gamma}<\sigma$ in Bruhat order.
Consider the set $\Delta^\vee_+\cap s_{\gamma}\Delta^\vee_+$. Analogously to the previous case the numbers of elements of this set
sent to $\Delta^\vee_-$ by $\sigma$ and $\sigma s_{\gamma}$ are equal. 
$|\Delta^\vee_- \cap s_\gamma \Delta^\vee_+|\leq\langle 2 \rho, \gamma \rangle -1$.
If there is not equality then we have no quantum edges labeled by $\gamma$.
The strict inequality is if and only if $\gamma$ is a short nonsimple root of subalgebra of rank $2$.
 I. e. the strict inequality holds for {\it long} coroots which are not linear
combination of {\it simple long} coroots.
So there exist an edge of graph iff $\sigma\left(\Delta^\vee_+ \cap s_\gamma \Delta^\vee_-\right)\subset \Delta^\vee_-$,
$\sigma s_\gamma\left(\Delta^\vee_+ \cap s_\gamma \Delta^\vee_-\right)\subset \Delta^\vee_+$. It is easy to see that this two conditions
are equivalent. Assume that there exist an element $\alpha \in \Delta^\vee_+ \cap s_\gamma \Delta^\vee_-$ such that
$\sigma(\alpha) \in \Delta^\vee_+$. Then the condition $ii)$ holds for $\alpha$ and $\beta=-s_\gamma(\alpha)$.
\end{proof}

\begin{definition}\label{betasWeyl}%\label{quasiweylmodules}
Let $\bar\beta=(\beta_1, \dots, \beta_r)$ be a sequence of affine coroots.
For $\sigma \in W$, $\lambda\in X_-$, the generalized Weyl module with characteristics
$W_{\sigma(\lambda)}(\bar\beta,m)$, $m=0,\dots,r$ is the cyclic $\fn^{af}$
module with a generator $v$ and the following relations:
$\fh\T t^kv=0$, $k>0$ and
\begin{gather*}\label{weylvanishings}
%(f_{\alpha}\otimes t) v=0,  \alpha \in \sigma(\Delta_-);\ (e_{-\alpha}\otimes 1) v=0, ~ \alpha \in \sigma(\Delta_-);\\
\widehat{\sigma}(f_{-\alpha}\T t)v=0,\\
\label{weylbound1s}
\widehat{\sigma}(e_\alpha)^{l_{\al,m}+1}v=0,
\end{gather*}
where $l_{\al,m}= -{\langle \alpha^\vee, \lambda \rangle-|\lbrace \beta_i|{\rm Re} \beta_i=-\alpha^\vee, i \leq m \rbrace|}.$
\end{definition}

\begin{rem}
If $m=0$, then $W_{\sigma(\lambda)}(\bar\beta,0)\simeq W_{\sigma(\lambda)}$. Now assume that $m=r$ and
the sequence of coroots $\bar\beta$ comes from a reduced decomposition of $t_{-\omega_i}$. Then according to Proposition
\ref{descriptionbeta}, part $a)$, we have an isomorphism
\[
W_{\sigma(\lambda)}(\bar\beta,r)\simeq W_{\sigma(\lambda+\omega_i)}.
\]
\end{rem}

\begin{example}
Let $\fg=\msl_3$ and let $\beta_1=-\al_1+\delta$, $\beta_2=-\al_1-\al_2+\delta$ (i.e. $\bar\beta$ comes from
the reduced decomposition of $t_{-\om_1}$, $\beta_1=\beta_1(t_{-\om_1})$, $\beta_2=\beta_2(t_{-\om_1})$).
Assume that $-\la=m_1\omega_1+m_2\omega_2$ and $m_1>0$.
Then we have the modules $W_{\sigma(\lambda)}(\bar\beta,0)$,
$W_{\sigma(\lambda)}(\bar\beta,1)$ and $W_{\sigma(\lambda)}(\bar\beta,2)$. The module  $W_{\sigma(\lambda)}(\bar\beta,0)$
is isomorphic to the generalized Weyl module $W_{\sigma(\lambda)}$. The defining relations for the module
$W_{\sigma(\lambda)}(\bar\beta,1)$
differ from the defining relations for $W_{\sigma(\lambda)}$ only by
\[
\widehat{\sigma}(e_{\al_1})^{m_1} v=0
\]
(no plus one in the exponent). Finally, the defining relations for the module $W_{\sigma(\lambda)}(\bar\beta,2)$
differ from the defining relations for $W_{\sigma(\lambda)}$ by two relations:
\begin{gather*}
\widehat{\sigma}(e_{\al_1})^{m_1} v=0,\\
\widehat{\sigma}(e_{\al_1+\al_2})^{m_1+m_2} v=0.
\end{gather*}
Hence $W_{\sigma(\lambda)}(\bar\beta,2)$ is isomorphic to $W_{\sigma(\lambda+\omega_1)}$.
\end{example}

For a (semi)simple Lie algebra $L$ we denote by $\fn^{af}(L)$ the Lie algebra $\fn^{af}$ attached to $L$,
$\fn^{af}(L)\subset \widehat L$ (if no confusion is possible, we omit $L$ an write simply $\fn^{af}$).

\begin{rem}
All the definitions above were given for a simple $\fg$. However, everything works fine in the semisimple case.
We only need this generalization in Lemma \ref{subsystemofrank2} below for $L$ of type $A_1\oplus A_1$.
\end{rem}

\begin{lem}\label{subsystemofrank2}
Let $\tau_1, \tau_2 \in \Delta_+$ be two roots from the roots system of $\fg$. Let $L_2$ be a semisimple Lie algebra with the root system
spanned by roots $\tau_1, \tau_2$.
For a $\fn^{af}(\fg)$-module
$W_{\sigma(\lambda)}(\bar\beta,m)$ we define $\fn^{af}(L_2)$-submodule
$M_2={\rm U}(\fn^{af}(L_2))v\subset W_{\sigma(\lambda)}(\bar\beta,m)$, where
$v$ is the cycic vector and $m$ satisfies $\sigma({\rm Re}\beta_{m+1})\in \mathbb{Z}\langle \tau^\vee_1, \tau^\vee_2\rangle$.
Then $M_2$ is a quotient of some $\fn^{af}(L_2)$ module of the form $W_{\widetilde \sigma(\widetilde \lambda)}(\widetilde\beta,\widetilde m)$,
where $\widetilde \sigma$, $\widetilde \lambda$, $\widetilde\beta$, $\widetilde m$ are parameters for $L_2$.
In addition,  $\sigma {\rm Re} \beta_{m+1}=
\widetilde\sigma{\rm Re} \widetilde{\beta}_{\widetilde m+1}$.
\end{lem}
\begin{proof}
Without loss of generality we assume that $\tau_1, \tau_2$ is the basis of $\mathbb{Z}\langle \tau_1, \tau_2\rangle \cap \Delta$.
If $L_2 \simeq A_1 \oplus A_1$, then
the claim is obvious. If $L_2 \simeq G_2$, then
$L_2 =\mathfrak g$ and hence there is nothing to prove.

We consider the root system $\sigma^{-1}\mathbb{Z}\langle \tau_1, \tau_2\rangle \cap \Delta$. Let $\eta_1, \eta_2$ be a basis of this
system such that $\eta_1, \eta_2 \in \Delta_+$ and $\eta_1,\eta_2$ are the simple roots in the root system
$\sigma^{-1}\mathbb{Z}\langle \tau_1, \tau_2\rangle \cap \Delta$.
Let
$\widetilde \sigma$ be the only element of the Weyl group of the root system $\mathbb{Z}\langle \tau_1, \tau_2\rangle \cap \Delta$
such that
$\widetilde \sigma^{-1}\sigma\eta_i\in \Delta_+$, $i=1,2$. Let $\widetilde \lambda$ be an anti-dominant weight for the Lie algebra
$L_2$ such that
$\langle \eta_i^\vee,\lambda \rangle=\langle \tau_i^\vee,\widetilde\lambda \rangle$.
If $m=0$, then we have the following relations in $W_{\sigma(\lambda)}(\bar\beta,m)$:
\[(\widehat\sigma f_{a_1\eta_1+a_2\eta_2})^{-\langle (a_1\eta_1+a_2\eta_2)^\vee,\lambda \rangle+1}v=0.\]
We rewrite this relation in terms of $M_2$:
\[(\widehat{\tilde\sigma} f_{a_1\tau_1+a_2\tau_2})^{-\langle (a_1\tau_1+a_2\tau_2)^\vee,\tilde\lambda \rangle+1}v=0.\]
Thus, $M_2$ is a quotient of $W_{\widetilde \sigma(\widetilde \lambda)}$.

Now we consider the case of general $m$. There are three possible cases: either $-{\rm Re} \beta_{m+1}$
is equal to one of the simple coroots of the Lie algebra of rank $2$ (i.e. to $\eta_i^\vee$), or to the sum $\eta_1^\vee+\eta_2^\vee$, or
$-{\rm Re} \beta_{m+1}= \eta_1^\vee + 2\eta_2^\vee$.

Let ${\rm Re} \beta_{m+1}=-\eta_i^\vee$.
Then using Corollary \ref{order}, $ii)$, $iii)$ we get for a root $\iota$:
\begin{gather*}
\text{if } \iota^\vee=\eta_1^\vee+\eta_2^\vee, \text{ then }  l_{\iota,m}=l_{\eta_1,m}+l_{\eta_2,m},\\
\text{if } \iota^\vee=\eta_1^\vee+2\eta_2^\vee \text{ then } l_{\iota,m}=l_{\eta_1,m}+2l_{\eta_2,m}.
\end{gather*}
Thus $M_2$ is a quotient of $W_{\widetilde\sigma(l_{\eta_1,m}\omega_1+l_{\eta_2,m}\omega_2)}$.

Now assume that $-{\rm Re} \beta_{m+1}=\eta_1^\vee+\eta_2^\vee$. Then using Corollary \ref{order}, $ii)$, we have that
\[
l_{-{\rm Re} \beta_{m+1}}=l_{\eta_1,m}+l_{\eta_2,m}+1.
\]
Then we obtain for $L_2 \simeq A_2$ the surjection
\[W_{\widetilde \sigma((l_{\eta_1,m}+1)\omega_1+l_{\eta_2,m}\omega_2)}(\bar\beta^1,1)\twoheadrightarrow M_2.\]
This completes the proof for $L_2 \simeq A_2$.

We are left with the case $L_2 \simeq C_2$, which is a direct consequence of Corollary \ref{order}, $iii)$.
\end{proof}

\subsection{The decomposition procedure}
Let us fix $i=1,\dots,n$ such that $\bra \la,\al_i^\vee\ket<0$ (i.e. $\omega_i$ shows up as a summand of $\la$).
In what follows we  assume that the sequence of coroots $\bar\beta^i=(\beta_1^i,\dots,\beta_r^i)$ come from a reduced decomposition of
$t_{-\omega_i}$, i.e. $\beta^i_j=\beta_j(t_{-\omega_i})$.
Now our strategy is as follows.
We first consider the sequence of surjections involving generalized Weyl modules with characteristics:
\[
W_{\sigma(\la)} = W_{\sigma(\la)}(\bar\beta^i,0)\to W_{\sigma(\la)}(\bar\beta^i,1)\to\dots\to
W_{\sigma(\la)}(\bar\beta^i,r)= W_{\sigma(\la+\om_i)}.
\]
In order to control the structure of $W_{\sigma(\la)}$ we need to describe the kernels
\begin{equation}\label{firststep}
{\rm ker} (W_{\sigma(\la)}(\bar\beta^i,m)\to W_{\sigma(\la)}(\bar\beta^i,m+1)).
\end{equation}
The kernel can be trivial or not. It is trivial if there is no edge $\sigma \to \sigma s_{{\rm Re} \beta_{m+1}}$
in the quantum Bruhat graph and non trivial otherwise. So our first step is to pick
a root $\beta^i_{m_1+1}$ such that there is an edge $\sigma \to \sigma s_{{\rm Re} \beta^i_{m_1+1}}$ in the QBG
and to pass to the kernel \eqref{firststep}. We note that we may also choose nothing at the first step (this corresponds
to the case $m_1=0$).
Now the second step is to describe the kernel of the surjection
$W_{\sigma(\la)}(\bar\beta^i,m_1)\to W_{\sigma(\la)}(\bar\beta^i,m_1+1)$.
We identify this kernel with the generalized Weyl module with characteristics of the form
$W_{\sigma_1(\la)}(\bar\beta^i,m_1+1)$ for $\sigma_1=\sigma s_{{\rm Re} \beta^i_{m_1+1}}\in W$. We have the sequence of surjections
\[
W_{\sigma_1(\la)}(\bar\beta^i,m_1+1)\to W_{\sigma_1(\la)}(\bar\beta^i,m_1+2)\to\dots\to
W_{\sigma_1(\la)}(\bar\beta^i,r)= W_{\sigma_1(\la+\om_i)}.
\]
Again, the kernel ${\rm ker}(W_{\sigma_1(\la)}(\bar\beta^i,m_2)\to W_{\sigma_1(\la)}(\bar\beta^i,m_2+1))$ is nontrivial
if there is an edge $\sigma_1 \to \sigma_1 s_{{\rm Re} \beta^i_{m_2+1}}$ in the QBG.
So our second step is to choose a root $\beta_{m_2+1}$, $m_2>m_1$ in such a way that there is a
path
\[
\sigma \to \sigma s_{{\rm Re} \beta^i_{m_1+1}}\to \sigma s_{{\rm Re} \beta^i_{m_1+1}} s_{{\rm Re} \beta^i_{m_2+1}}
\]
in the QBG. Each such a path gives rise to a generalized Weyl module with characteristics.
We proceed further, making totally $r$ steps (note that at each step we may skip making a choice of
a root $\beta^i_j$). Then after the $r$-th step we obtain the decomposition procedure, representing the initial module
$W_{\sigma(\la)}$ via the set of subquotients.  We have the following important features:
\begin{itemize}
\item All the subquotients are of the form $W_{\kappa(\la+\om_i)}$ for some $\kappa\in W$.
\item The subquotients are labeled by the paths in the QBG of length at most $r$ of the form
\[
\sigma \to \sigma s_{{\rm Re} \beta^i_{m_1+1}}\to \sigma s_{{\rm Re} \beta^i_{m_1+1}} s_{{\rm Re} \beta^i_{m_2+1}}\to
\dots\to \sigma s_{{\rm Re} \beta^i_{m_1+1}} \dots s_{{\rm Re} \beta^i_{m_p+1}}
\]
for some $0\le m_1 < \dots <m_p<r$, $p<r$.
\end{itemize}
We prove that the whole picture can be seen as a representation theoretic interpretation of the combinatorial construction from
Theorem \ref{combinatorialdecomposition}.

In the next theorem we describe the properties of the modules $W_{\kappa(\lambda)}(\bar\beta^i,m)$ in terms of the quantum Bruhat graph.
Recall
\[
l_{\al,m}= -{\langle \alpha^\vee, \lambda \rangle-|\lbrace \beta_k^i|-{\rm Re} \beta_k^i=\alpha^\vee, k \leq m \rbrace|}.
\]
\begin{thm}\label{stepofdecomposition}
Let $\bar\beta^i=(\beta_1^i,\dots,\beta_r^i)$ be a sequence of $\beta$'s for some reduced decomposition of the element $t_{-\omega_i}$.
Then we have:

$i)$ Assume there is no edge $\sigma \stackrel{{\rm Re} \beta_{m+1}^i}{\longrightarrow} \sigma s_{{\rm Re} \beta_{m+1}^i}$, then
\[W_{\sigma(\lambda)}(\bar\beta^i,m) \simeq W_{\sigma(\lambda)}(\bar\beta^i,m+1).\]

$ii)$ Assume there is an edge $\sigma \stackrel{{\rm Re} \beta_{m+1}^i}{\longrightarrow} \sigma s_{{\rm Re}\beta_{m+1}^i}$.
Then for $\al^\vee=-{\rm Re} \beta_{m+1}^i$ we have an exact sequence
\[
0\to {\rm U}(\fn^{af}) (\widehat{\sigma} e_\al)^{l_{\al,m}} v\to
 W_{\sigma(\lambda)}(\bar\beta^i,m)\to
W_{\sigma(\lambda)}(\bar\beta^i,m+1)\to 0.
\]

$iii)$  Assume there is an edge $\sigma \stackrel{{\rm Re} \beta_{m+1}^i}{\longrightarrow} \sigma s_{{\rm Re}\beta_{m+1}^i}$. Then
for $\al^\vee=-{\rm Re} \beta_{m+1}^i$ there exists a surjection
\[
W_{{\sigma}s_\al(\lambda)}(\bar\beta^i,m+1) \twoheadrightarrow
{\rm U}(\fn^{af}) (\widehat{\sigma}e_\al)^{l_{\al,m}} v.
\]
%\begin{gather*}
%{\rm U}(\fn^{af}) \widehat{\sigma}(e_{{\rm Re} \beta_{m+1}})^{l_{m+1}} v\twoheadleftarrow
%W_{{\sigma}s_{{\rm Re} \beta_{m+1}}(\lambda)}(\bar\beta^i,m+1),\\
%W_{\sigma(\lambda)}(\bar\beta^i,m+1)\simeq W_{\sigma(\lambda)}(\bar\beta^i,m)/
%{\rm U}(\fn^{af}) \widehat{\sigma}(e_{{\rm Re} \beta_{m+1}})^{l_{m+1}} v.
%\end{gather*}

$iv)$ We have an exact sequence
\begin{equation}\label{genWeylasquotient}
0\to \sum_{\sigma \stackrel{{\rm Re} \beta_{m+k}^i}{\longrightarrow}
\sigma s_{{\rm Re} \beta_{m+k}^i}}{\rm U}(\fn^{af}) 
\widehat{\sigma}(e_{-{\rm Re} \beta_{m+k}^{i^\vee}})^{l_{-{\rm Re} \beta_{m+k+1}^{i^\vee},m+k}}v\to
W_{\sigma(\lambda)}(\bar\beta^i,m)\to W_{\sigma(\lambda+\omega_i)}\to 0
\end{equation}
(the sum is taken over all $k\ge 1$ such that the edge $\sigma \stackrel{{\rm Re} \beta_{m+k}^i}{\longrightarrow}
\sigma s_{{\rm Re} \beta_{m+k}^i}$
does exist in the quantum Bruhat graph).
%W_{\sigma(\lambda)}(\bar\beta^i,s)/\left(\sum_{w \stackrel{{\rm Re} \beta_{s+k}}{\longrightarrow}
%\sigma s_{{\rm Re} \beta_{s+k}}}{\rm U}(\fn^{af}) \widehat{\sigma}(e_{({\rm Re} \beta_{s+k})v})^{l_{s+k}}v\right)
%\simeq W_{\sigma(\lambda)}.
\end{thm}
\begin{proof}
Let us prove $i)$.
Assume that there is no edge $\sigma \stackrel{{\rm Re} \beta_{m+1}^i}{\longrightarrow} \sigma s_{{\rm Re} \beta_{m+1}^i}$.
Then according to Lemma \ref{edgesinQBG} we have a rank two algebra $L_2$ such that
$\sigma ({\rm Re} \beta_{m+1}^i)$ is a root for $L_2$. So
we either have such
$\tau, \eta \in \Delta_+^\vee$, $\tau, \eta \neq \gamma$ satisfying
\begin{gather*}
\tau + \eta=\frac{\langle \tau, {\rm Re}\beta_{m+1}^i \rangle}{\langle {\rm Re}\beta_{m+1}^i, 
{\rm Re}\beta_{m+1}^i \rangle}{\rm Re}\beta_{m+1}^i,\\
\widehat{\sigma}(\tau) + \widehat{\sigma}(\eta)=\frac{\langle \tau, {\rm Re}\beta_{m+1}^i \rangle}
{\langle {\rm Re}\beta_{m+1}^i, {\rm Re}\beta_{m+1}^i \rangle}\widehat{\sigma}({\rm Re}\beta_{m+1}^i)
\end{gather*}
or
${\rm Re} \beta_{m+1}^i$ is a nonsimple short root of some subalgebra of rank $2$ and $\sigma(-{\rm Re} \beta_{m+1}^i)\in \Delta_-^\vee$.
Now the claim follows from Lemma \ref{subsystemofrank2} and the rank two results from Section \ref{LRC}.

Now assume that there exists an edge $\sigma \stackrel{{\rm Re} \beta_{m+1}^i}{\longrightarrow} \sigma s_{{\rm Re} \beta_{m+1}^i}$.
Then part $ii)$ follows directly from Definition \ref{betasWeyl}.
Let us prove $iii)$.
We have to show that for $\al^\vee=-{\rm Re} \beta_{m+1}^i$ the following relations hold:
\begin{equation}
(\widehat{\sigma s_\al} e_\gamma)^{l_{\gamma,m+1}+1}
(\widehat{\sigma} e_\al)^{l_{\al,m}} v=0,  \gamma \in \Delta_+.
\end{equation}
Let us consider the Lie algebra with the root system spanned by the roots $\al^\vee$ and ${\rm Re} \beta_{m+1}^i$.
Our claim now follows from Lemma \ref{subsystemofrank2} and direct computations from Section \ref{LRC}.

Finally, part $iv)$ is an immediate corollary from Definition \ref{betasWeyl} and Lemma \ref{descriptionbeta}, $a)$.
\end{proof}

\begin{cor}\label{decompositioninequality}
Let $\la\in X_-$ and $\sigma\in W$. Then
\[{\rm ch} W_{\sigma(\lambda-\om_i)}\le \sum_{p\in \mathcal{QB}(\sigma,\lambda,\bar\beta^{i,\lambda})}
q^{{\rm deg}({\rm qwt}(p))} {\rm ch} W_{dir({\rm end}(p))(\la)}^{wt ({\rm end} (p))},\]
\[{\rm ch} W_{\sigma(\lambda)}\le C_\sigma^{t_{\lambda}},\]
where inequalities mean the coefficient-wise inequalities.
\end{cor}
\begin{proof}
Using Theorem \ref{stepofdecomposition} we obtain that $W_{\sigma(\lambda-\om_i)}$ can be decomposed to subquotients isomorphic
to quotients of $W_{dir({\rm end}(p))(\lambda)}$, $p \in \mathcal{QB}(\sigma,\lambda,\bar\beta^{i,\lambda})$.
Therefore we only need to prove that the cyclic vectors of these modules have the needed weights. Let $v$ be
a cyclic vector in $W_{\sigma(\lambda)}(\bar\beta^i,m)$, $\al^\vee=-{\rm Re}\beta_{m+1}^i$ and
$v_1=(\widehat\sigma e_\al)^{l_{\al,m}}v$. Then we have:
\[
x^{wt(v_1)}q^{{\rm deg} v_1}=\begin{cases}
x^{wt(v)+l_{\al,m}\sigma(\al)}q^{{\rm deg} v},~if~  \sigma(\al)  \in \Delta_+;\\
x^{wt(v)+l_{\al,m}\sigma(\al)}q^{{\rm deg} v+l_{\al,m}},
~if ~\sigma (\al) \in \Delta_-.
\end{cases}
\]
Now let us show that $l_{-{\rm Re}\beta_{m+1}^{i^\vee},m}=\deg\beta_{m+1}^i$. First, assume that $m=0$.
Then $l_{-{\rm Re}\beta_j^\vee,0}=-\bra {\rm Re}\beta_j,\la\ket$. Now if $\beta_j$ is the first root in $\bar\beta^i$ with the fixed
real part, then Proposition \ref{descriptionbeta}, $a)$ and Lemma \ref{firstbetas} give us the needed equality.
Now assume $m>0$.
Let $\beta_{j_a}$ be the subsequence of $\bar\beta^i$ such that ${\rm Re}\beta_{j_a}={\rm Re}\beta$. Then we have that
$\deg\beta_{j_{a+1}}=\deg\beta_{j_{a}}-1$, $l_{-{\rm Re}\beta^\vee_{j_{a+1}},j_{a+1}-1}=l_{-{\rm Re}\beta^\vee_{j_{a}},j_{a}-1}-1$.
Thus the $q$-component of
the weights of the cyclic elements of subquotients are equal to $\deg({\rm qwt}(p))$.

Now we need to compare the finite weights coming from combinatorial and representation theoretic constructions.
Let $\tau={\rm dir}({\rm end}(p))$.
Assume that the real part of the weight of a vector $u$ is equal to $\tau(\la)$. Then if
$\tau{\rm Re} \beta_{m+1}^{i^\vee} \in \Delta_-$, then
the real part of the weight of $u_1=(\widehat\tau e_{-{\rm Re}\beta_{m+1}^{i^\vee}})^{l_{-{\rm Re}\beta_{m+1}^{i^\vee},m}}u$
is equal to $\tau(\la)+{\rm deg}(\beta_{m+1}^i)\tau({\rm Re}\beta_{m+1}^{i^\vee})$.
However:
\[\tau(\la)+{\rm deg}(\beta_{m+1}^i)\tau({\rm Re}\beta_{m+1}^{i^\vee})=wt(end(p)s_{\beta_{m+1}^i}).\]
Indeed, $end(p)=t_{\tau(\lambda)}\tau$ and
\begin{multline*}
t_{\tau(\lambda)}\tau s_{\beta_{m+1}^i}=t_{\tau(\lambda)}\tau t_{{\rm deg}(\beta_{m+1}^i){\rm Re}\beta_{m+1}^{i^\vee}}
s_{{\rm Re}\beta_{m+1}^{i}}=\\
t_{\tau(\lambda)} t_{{\rm deg} (\beta_{m+1}^i)\tau({\rm Re}\beta_{m+1}^{i^\vee})}\tau s_{{\rm Re}\beta_{m+1}^i}.
\end{multline*}
Analogously we obtain the claim for $\tau {\rm Re} \beta_{m+1}^{i^\vee} \in \Delta_+$.
\end{proof}

We denote by $E_\la(1,1,0)$ the specialization of the Macdonald polynomials at $t=0$, $q=1$ and all $x_i=x^{\om_i}=1$.
\begin{rem}\label{fw}
In the following theorem we use that $\dim W(\omega_i)=E_{w_0 \omega_i}(1,1,0)$ for all fundamental weights.
This is a very special case of \cite{CI}, Theorem 4.2 (see also  \cite{LNSSS3,N}). Indeed, Theorem 4.2,\cite{CI} 
claims that for any dominant weight $\mu$ 
the character of the Weyl module $W(\mu)$ is equal to the value of symmetric Macdonald polynomial $P_{\mu}$ specialized at
$t=0$ (in \cite{CI} the symmetric Macdonald polynomials are labeled by the anti-dominant weights, so in the Chari-Ion notation
the character is expressed in terms of $P_{w_0\mu}$). Thanks to \cite{I}, Theorem 4.2, one has
$P_\mu(x,q,0)=E_{w_0\mu}(x,q,0)$, which implies $\ch W(\mu)=E_{w_0 \mu}(x,q,0)$. 
We note that the Chari-Ion theorem addresses the case of general dominant weights.
For our purposes we only need the $x=1$ specialization of their theorem and only for fundamental weights.  
\end{rem}
\begin{thm}\label{ineq=eq}
The inequalities of Corollary \ref{decompositioninequality} are in fact the equalities.
\end{thm}
\begin{proof}
According to Remark \ref{fw} $\dim W(\omega_i)=E_{w_0 \omega_i}(1,1,0)$ for all fundamental weights $\omega_i$.
We note that for dominant weights $\nu, \mu$ we have (see \cite{I},\cite{N}):
\begin{equation}\label{mult}
\dim W(\nu+\mu)=\dim W(\nu)\cdot \dim W(\mu).
\end{equation}
Moreover, for the specialization at $q=1$ of the {\it symmetric} Macdonald polynomials $P_{\lambda}(x,1,t)$ we have
$P_{\nu+\mu}(x,1,t)=P_{\nu}(x,1,t)\cdot P_{\mu}(x,1,t)$ and for any dominant $\lambda$ there is an equality
$P_{\lambda}(x,q,0)=E_{w_0(\lambda)}(x,q,0)$ (see \cite{I}, Theorem 4.2).
Hence we have for any dominant $\lambda$:
\[\dim W(\lambda)=E_{w_0 \lambda}(1,1,0).\]

We know that for any $\sigma \in W$  the following holds:
\[
\dim W_{\sigma \lambda}\geq \dim W(\lambda)=E_{w_0 \lambda}(1,1,0)
\]
(Lemma \ref{lowerbound} plus \eqref{mult}).
Note that $E_{w_0 \lambda}(1,1,0)$ is the number of paths of type $t_{w_0(\lambda)}$ in the quantum Bruhat graph starting at
the identity element of $W$.
We also know that  $\dim W_{\sigma(w_0\omega_i)}$ is less or equal than the
number of paths in the quantum Bruhat graph of type $\bar \beta^i$ starting at the point $\sigma$ (for any $\sigma\in W$).
Assume that for some $\sigma \in W$ and a fundamental weight $\omega_i$ the strict inequality
$\dim W_{\sigma(w_0\omega_i)}>E_{w_0 \omega_i}(1,1,0)$ holds.
For a decomposition $\sigma=s_{j_1}\dots s_{j_u}$ we define
$\lambda=\omega_{j_1}+\dots + \omega_{j_u}+\omega_i$. Then using Theorem \ref{combinatorialdecomposition} $(u+1)$ times we obtain:
\begin{multline}
E_{w_0 \lambda}(1,1,0)=
\sum_{p_1 \in \mathcal{QB}({\rm id}, \lambda-\om_{j_1}, \bar\beta^{{j_1},\lambda})}
\sum_{p_2 \in \mathcal{QB}({end p_1}, \lambda-\omega_{j_1}-\omega_{j_2}, \bar\beta^{{j_2},-\omega_{j_1}})}\dots\\
\sum_{p_u \in \mathcal{QB}({end p_{u-1}}, \omega_i, \bar\beta^{j_u,\omega_{j_u}+\omega_i})}
\sum_{p_{u+1} \in \mathcal{QB}({end p_{u}}, 0, \bar\beta^{i,\omega_i})}1.
\end{multline}
Every time at the $k$-th sum we sum up at least $E_{w_0 \omega_{j_k}}(1,1,0)$ summands. Indeed, the number of summands is
not smaller than $\dim W_{\kappa(\la)}$ for some $\kappa\in W$. But we also know that
\[
\dim W_{\kappa(\la)}\ge \dim W({\la})=E_{w_0 \omega_{j_k}}(1,1,0).
\]
Therefore if even once we sum up strictly more than
$E_{w_0 \omega_{j_k}}(1,1,0)$ summands, then $\dim W(\lambda)>\prod_{k=1}^m \dim W(\omega_{j_m})\cdot \dim W(\omega_{i})$,
which contradicts \eqref{mult}. Using Corollary \ref{order}, $i)$ we have that ${\rm Re} \beta^j_1=-\alpha^\vee_j$.
For any $\kappa \in W$ and any
simple root $\alpha_j$ there exist an edge $\kappa \stackrel{\alpha^\vee_j}{\longrightarrow} \kappa s_j$ in the QBG. Therefore
in the last summation we at least once have $\sum_{p_{u+1} \in \mathcal{QB}(\sigma, 0, \bar\beta^{i,\omega_i})}1$,
i.e. ${\rm dir} (end(p_u))=\sigma$.
Therefore for any $\sigma \in W$ we have exactly $\dim W(\omega_{i})$ paths of type $\bar \beta^i$.
So we conclude that for any dominant $\mu=\sum_{k=1}^N \om_{j_k}$ one has
\[
{\dim} W_{\sigma(w_0\mu)} \le C_\sigma^{t_{w_0\mu}}(1,1)=\prod_{k=1}^N \dim W(\omega_{j_k}).
\]
Now using Lemma \ref{lowerbound} we obtain $\dim W_{\sigma(\lambda)} =\prod_{k=1}^N \dim W(\omega_{j_k})$.
\end{proof}

\begin{cor}
Let $\la$ be an anti-dominant weight, $\sigma\in W$. Then
${\rm ch} W_{\sigma(\la)}=C_{\sigma}^{t_{\la}}$.
\end{cor}

As a consequence, we obtain an alternative proof of the following claim (see \cite{I} for $\mathfrak g$ of types $A$, $D$, $E$ and \cite{CI} for general simple
Lie algebras).
\begin{cor}\label{Iongeneralization}
Let $\lambda$ be a dominant weight. Then for arbitrary simple $\fg$
\[E_{w_0(\lambda)}(x,q,0)=\ch W(\lambda).\]
\end{cor}

We also obtain a representation-theoretic interpretation of the specialization of nonsymmetric Macdonald polynomials at $t=\infty$.

\begin{cor}\label{representationinterpretationinfinity}
Let $\lambda$ be an anti-dominant weight. Then:
\[w_0E_{\lambda}(x,q^{-1},\infty)=\ch W_{w_0\lambda}.\]
\end{cor}

\begin{rem}
In \cite{No} the author proves the relationship between the graded characters of generalized
Weyl modules and those of certain quotients of Demazure submodules of level $0$ extremal
weight modules over quantum affine algebras.
\end{rem}

\section{Low rank cases}\label{LRC}
\subsection{Type $A_1$}
Let $\fg=\msl_2$. The QBG has two vertices ${\rm id}$ and $s$ and two arrows: from ${\rm id}$ to $s$ and backwards.
We have two types of generalized Weyl modules, corresponding to $\sigma={\rm id}$ and to $\sigma=s$.
There is only one fundamental weight $\om_1$ and the sequence $\bar\beta^1$ consists of one element $\beta_1^1=-\al+\delta$.
The modules of the form $W_{\la}$, $\la=-n\omega$, $n\ge 0$ are isomorphic to the level one Demazure modules.
The module $W_{-n\omega}$, $n\ge 0$ is defined by the relations
\[
(e\T 1)^{n+1} v_{-n}=0,\ (f\T t) v_{-n}=0,\ h\T t^k v_{-n}=0, k>0.
\]
Now the modules $W_{n\omega}$, $n>0$ are defined by the relations
\[
e\T 1 v_{n}=0, \ (f\T t)^{n+1}v_{n}=0,\ h\T t^k v_{n}=0, k>0.
\]
One has $\dim W_{n\omega}=\dim W_{-n\omega}=2^n$.

Since $\bar\beta^1$ consists of a single root, the Weyl modules with characteristics are isomorphic to
the classical Weyl modules. Namely,
\[
W_{\sigma\la}(\bar\beta^1,0)\simeq W_{\sigma\la},\ W_{\sigma\la}(\bar\beta^1,1)\simeq W_{\sigma(\la-\omega)}.
\]
We have the following properties of the generalized Weyl modules:
\begin{gather*}
W_{-n\omega}\supset {\rm U}(\fn^{af})(e\T 1)^nv_{-n}\simeq W_{(n-1)\omega},\\
W_{-n\omega}/ {\rm U}(\fn^{af})(e\T 1)^nv_{-n}\simeq W_{-(n-1)\omega}.
\end{gather*}
Similarly, one has
\begin{gather*}
W_{n\omega}\supset {\rm U}(\fn^{af})(f\T t)^nv_{n}\simeq W_{-(n-1)\omega},\\
W_{n\omega}/ {\rm U}(\fn^{af})(f\T t)^nv_{n}\simeq W_{(n-1)\omega}.
\end{gather*}

\subsection{Type $A_2$}
The goal of this section is to describe explicitly the structure of the generalized Weyl modules for $\fg=\msl_3$.
More precisely, we prove Theorem \ref{stepofdecomposition} in type $A_2$. The quantum Bruhat graph of type $A_2$
looks as follows
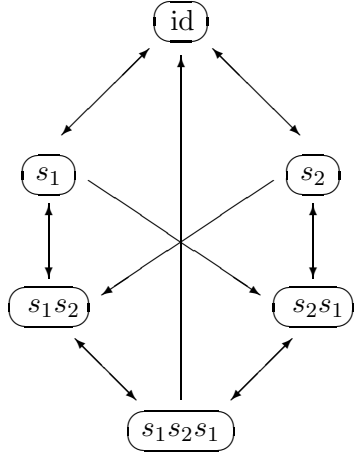
\begin{figure}
\begin{picture}(200,200)
\put(100,175){\oval(20,15)}
\put(96,172){${\rm id}$}

\put(50,117){\oval(20,15)}
\put(45,115){$s_1$}

\put(150,117){\oval(20,15)}
\put(145,115){$s_2$}

\put(50,67){\oval(30,15)}
\put(42,65){$s_1s_2$}

\put(150,67){\oval(30,15)}
\put(142,65){$s_2s_1$}

\put(100,20){\oval(40,15)}
\put(85,18){$s_1s_2s_1$}

\put(71,44){\vector(1,-1){12}}
\put(71,44){\vector(-1,1){11}}

\put(130,43){\vector(1,1){12}}
\put(130,43){\vector(-1,-1){11}}

\put(50,88){\vector(0,1){18}}
\put(50,88){\vector(0,-1){10}}

\put(150,88){\vector(0,1){18}}
\put(150,88){\vector(0,-1){10}}

\put(65,142){\vector(1,1){23}}
\put(65,142){\vector(-1,-1){10}}

\put(135,142){\vector(-1,1){23}}
\put(135,142){\vector(1,-1){10}}

\put(65,115){\vector(3,-2){65}}

\put(135,115){\vector(-3,-2){65}}

\put(100,32){\vector(0,1){130}}
\end{picture}
\caption{$QBG$ of type $A_2$}
\end{figure}

We consider the module $W_{\sigma(\lambda)}$, where $\la=-n_1\om_1-n_2\om_2$ and $\sigma$ is an element in the
permutation group $S_3$. We assume that $n_1$ is positive and fix $\beta_1=-\al_1+\delta$, $\beta_2=-\al_1-\al_2+\delta$,
so $\beta_1=\beta_1(t_{-\om_1})$, $\beta_2=\beta_2(t_{-\om_1})$ (the decomposition procedure with
respect to $\omega_2$ is very similar). Since the sequence $\bar\beta$ is fixed, we omit
$\bar\beta$ when talking about the generalized Weyl modules with characteristics and write simply
$W_{\sigma(\lambda)}(m)$ instead of $W_{\sigma(\lambda)}(\bar\beta,m)$.
We have the following sequence of surjections of ${\rm U}(\fn^{af})$-modules:
\begin{equation}\label{surjchar}
W_{\sigma(\lambda)}\simeq W_{\sigma(\lambda)}(0)\twoheadrightarrow W_{\sigma(\lambda)}(1) \twoheadrightarrow W_{\sigma(\lambda)}(2)
\twoheadrightarrow W_{\sigma(\lambda)}(2)\simeq  W_{\sigma(\lambda+\om_1)}.
\end{equation}
We use the notation:
\[
e_1=e_{\al_1},\ e_2=e_{\al_2},\ e_{12}=e_{\al_1+\al_2}
\]
and similarly for $f_{\al}$. We also denote the reflection in $S_3$ by $s_1$, $s_2$ and $s_{12}$.

{\bf Case $1$.} Let $\sigma={\rm id}$. Then the relations in $W_{\sigma(\lambda)}$ are of the form $h\T t^k v=0$, $k>0$ and
\[
e_1^{n_1+1}v=e_2^{n_2+1}v=e_{12}^{n_1+n_2+1}v=0,\ f_\al\T t v=0.
\]
Let us consider the sequence \eqref{surjchar}.

First of all, there is no edge
${\rm id} \stackrel{\al_{12}}{\longrightarrow} s_{12}$ in the quantum Bruhat graph. Therefore, we have to show
(Theorem \ref{stepofdecomposition}, $i)$) that
the map $W_{\lambda}(1) \twoheadrightarrow W_{\lambda}(2)$ is an isomorphism. Indeed, the only difference
between the defining relations is $e_{12}^{n_1+n_2+1}v=0$ in $W_{\lambda}(1)$ vs $e_{12}^{n_1+n_2}v=0$ in $W_{\lambda}(2)$.
However, we have the relations $e_1^{n_1}v=0$ and $e_2^{n_2+1}v=0$ in $W_{\lambda}(1)$, which imply
that $e_{12}^{n_1+n_2}v=0$ already in $W_{\lambda}(1)$.

Second let us consider the map $W_{\lambda}(0) \twoheadrightarrow W_{\lambda}(1)$. Obviously, the kernel of this map is given by
${\rm U}(\fn^{af})e_1^{n_1}v$ (Theorem \ref{stepofdecomposition}, $ii)$). We want to prove (Theorem \ref{stepofdecomposition}, $iii)$)
that there is a surjective homomorphism
\[
W_{s_1(\lambda)}(1)\to {\rm U}(\fn^{af})e_1^{n_1}v.
\]
In other words, we need to prove the following equalities in $W_{\lambda}$:
\begin{gather*}
e_1e_1^{n_1}v=0,\ (f_2\T t)e_1^{n_1}v=0,\ (f_{12}\T t)e_1^{n_1}v=0,\ (\fh\T t\bK[t])e_1^{n_1}v=0,\\
(f_1\T t)^{n_1}e_1^{n_1}v=0,\ e_{12}^{n_2+1}e_1^{n_1}v=0,\ e_2^{n_1+n_2+1}e_1^{n_1}v=0.
\end{gather*}
The relations in the first line  are obvious. Now let us consider the second line relations.
The first equality follows from the type $A_1$ picture and the second and third relations obviously hold
in the irreducible $\msl_3$-module $V_\la$ and hence in $W_\la$ as well.

So the kernel of the map $W_{\lambda}(0) \twoheadrightarrow W_{\lambda}(1)\simeq W_{\lambda+\om_1}$ is covered by $W_{s_1(\lambda)}(1)$
(in fact, the covering is an isomorphism as we prove below). To finalize the proof, we consider the surjection
\[
W_{s_1(\lambda)}(1)\to W_{s_1(\lambda)}(2).
\]
The kernel of this surjection is given by ${\rm U}(\fn^{af})e_2^{n_1+n_2}v$. We want to show that there is a surjective map
\[
W_{s_1s_{12}(\lambda)}(2)\to {\rm U}(\fn^{af})e_2^{n_1+n_2}v.
\]
So we have to show that the following equalities hold in $W_{s_1(\lambda)}(1)$:
\begin{gather}\label{s1(1)1}
e_2e_2^{n_1+n_2}v=0,\ e_{12}e_2^{n_1+n_2}v=0,\ (f_1\T t)e_2^{n_1+n_2}v=0,\ \fh\T t\bK[t]v=0,\\
\label{s1(1)2}
e_1^{n_2+1}e_2^{n_1+n_2}v=0,\ (f_2\T t)^{n_1+n_2}e_2^{n_1+n_2}v=0,\ (f_{12}\T t)^{n_1}e_2^{n_1+n_2}v=0.
\end{gather}
Recall the defining relations in $W_{s_1(\lambda)}(1)$:
\begin{gather*}
e_1v=0,\ (f_2\T t)v=0,\ (f_{12}\T t)v=0,\ \fh\T t\bK[t]=0,\\
(f_1\T t)^{n_1}v=0,\ e_{12}^{n_2+1}v=0,\ e_2^{n_1+n_2+1}v=0.
\end{gather*}
The relations \eqref{s1(1)1} can be derived easily (for example,
$(f_1\T t)e_2^{n_1+n_2}v$ is proportional to $(f_{12}\T t)e_2^{n_1+n_2+1}v$). Now let us derive the relations
\eqref{s1(1)1}. The relation $e_1^{n_2+1}e_2^{n_1+n_2}v=0$ can be obtained by commuting
$e_1^{n_2+1}$ through $e_2^{n_1+n_2}$ and using the relations $e_1v=0$ and $e_{12}^{n_2+1}v=0$.
The relation $(f_2\T t)^{n_1+n_2}e_2^{n_1+n_2}v=0$ follows from the $A_1$ case.
Finally, the relation  $(f_{12}\T t)^{n_1}e_2^{n_1+n_2}v=0$ can be obtained by commuting
$(f_{12}\T t)^{n_1}$ through $e_2^{n_1+n_2}$ and using the relations $(f_{12}\T t)v=0$,
$(f_1\T t)^{n_1}v=0$.

So we conclude, that the module $W_{\sigma(\lambda)}$ can be decomposed into three subquotients. Each subquotients is a quotient
of some $W_{\kappa(\la+\om_1)}$ for some $\kappa\in S_3$. By induction on $n_1+n_2$, the dimension of each subquotient does not exceed
$3^{n_1+n_2-1}$. Hence $\dim W_{\la}\le 3^{n_1+n_2}$. Since the opposite inequality always holds, we obtain that
$\dim W_{\la}=3^{n_1+n_2}$ and all the subquotient are of the form  $W_{\kappa(\la+\om_1)}$.

Now one easily checks that the cases of $\sigma=s_1s_2$ and $\sigma=s_2s_1$ are equivalent to the case
$\sigma={\rm id}$, since the three-dimensional nilpotent subalgebra, formed by the root operators $e_\al$ and $f_\al\T t$,
acting nontrivially on $v$, is isomorphic to the Heisenberg algebra.

{\bf Case $2$.} Let us work out the opposite case, i.e. when $\sigma=s_{\al_1+\al_2}=s_{12}$ is the longest element. Then
the relations in $W_{\sigma(\lambda)}$ are of the following form
\[
(f_1\T t)^{n_2+1}v=0, \ (f_2\T t)^{n_2+1}v=0,\ (f_{12}\T t)^{n_1+n_2+1}v=0,\ e_\al v=0.
\]

We have both edges
${\rm s_{12}} \stackrel{\al_{12}}{\longrightarrow} {\rm id}$ and ${\rm s_{12}} \stackrel{\al_1}{\longrightarrow} {\rm s_{12}s_1}$
in the quantum Bruhat graph. Therefore, we have to describe the kernels of the maps $W_{s_{12}(\lambda)}(0) \twoheadrightarrow W_{s_{12}(\lambda)}(1)$ and
$W_{s_{12}(\lambda)}(1) \twoheadrightarrow W_{s_{12}(\lambda)}(2)$.

First, let us consider the map $W_{s_{12}(\lambda)}(0) \twoheadrightarrow W_{s_{12}(\lambda)}(1)$. Obviously, the kernel of this
map is given by
${\rm U}(\fn^{af})(f_2\T t)^{n_1}v$ (Theorem \ref{stepofdecomposition}, $ii)$). We want to prove (Theorem \ref{stepofdecomposition}, $iii)$)
that there is a surjective homomorphism
\[
W_{s_{12}s_1(\lambda)}(1)\to {\rm U}(\fn^{af})(f_2\T t)^{n_1}v.
\]
In other words, we need to prove the following equalities in $W_{s_{12}(\la)}$:
\begin{gather*}
e_1(f_2\T t)^{n_1}v=0,\ e_{12}(f_2\T t)^{n_1}v=0,\\
(f_{2}\T t)(f_2\T t)^{n_1}v=0,\ (\fh\T t\bK[t])(f_2\T t)^{n_1}v=0
\end{gather*}
(these are obvious) and
\[
e_2^{n_1}(f_2\T t)^{n_1}v=0,\ (f_{12}\T t)^{n_2+1}(f_2\T t)^{n_1}v=0,\ (f_1\T t)^{n_1+n_2+1}(f_2\T t)^{n_1}v=0.
\]
The first equality follows from the type $A_1$ picture. The second relation comes from
the equality  $e_1^{n_1}(f_{12}\T t)^{n_1+n_2+1}v=0$. To prove the third relation we move  $(f_2\T t)^{n_1}$
to the left in the expression $(f_1\T t)^{n_1+n_2+1}(f_2\T t)^{n_1}$. All the terms in the resulting sum contain
the factor $(f_1\T t)^i$, $i>n_2$ on the very right and hence vanish being applied to $v$ in $W_{s_{12}(\lambda)}$.

The second step is to consider the map $W_{s_{12}(\lambda)}(1) \twoheadrightarrow W_{s_{12}(\lambda)}(2)$.
Obviously, the kernel of this map is given by
${\rm U}(\fn^{af})(f_{12}\T t)^{n_1+n_2}v$ (Theorem \ref{stepofdecomposition}, $ii)$). We want to
prove the existence of the surjective homomorphism
\[
W_{\lambda}(2)\to {\rm U}(\fn^{af})(f_{12}\T t)^{n_1+n_2}v.
\]
In other words, we need to prove the following equalities in $W_{s_{12}(\la)}(1)$:
\begin{gather*}
(f_1\T t)(f_{12}\T t)^{n_1+n_2}v=0,\ (f_2\T t)(f_{12}\T t)^{n_1+n_2}v=0,\\
(f_{12}\T t)(f_{12}\T t)^{n_1+n_2}v=0,\ (\fh\T t\bK[t])(f_2\T t)^{n_1}v=0
\end{gather*}
(these are obvious) and
\[
e_1^{n_1}(f_{12}\T t)^{n_1+n_2}v=0,\ e_2^{n_2+1}(f_{12}\T t)^{n_1+n_2}v=0,\ e_{12}^{n_1+n_2}(f_{12}\T t)^{n_1+n_2}v=0.
\]
The third equality follows from the type $A_1$ picture. The first relation can be proved by commuting
$e_1^{n_1}$ to the right through $(f_{12}\T t)^{n_1+n_2}$, since $(f_2\T t)^{n_1}v=0$ in $W_{s_{12}(\la)}(1)$.
The second relation is obtained in the same way.

Now our last step is to consider the module $W_{s_{12}s_1(\lambda)}(1)$. We are interested in the surjection
$W_{s_{1}s_2(\lambda)}(1)\to W_{s_{1}s_2(\lambda)}(2)$ ($s_{12}s_1=s_1s_2$).
Since there is no edge of the form ${\rm s_1s_2} \stackrel{\al_{12}}{\longrightarrow} s_2$ in the QBG, we need to prove that
the surjection above is in fact an isomorphism. In other words, we have to show that the relation
$(f_1\T t)^{n_1+n_2}v=0$ hold in $W_{s_{1}s_2(\lambda)}(1)$. We have the following relations in $W_{s_{1}s_2(\lambda)}(1)$:
$e_2^{n_1}v=0$ and $(f_{12}\T t)^{n_2+1}v=0$. Since $[f_{12}\T t,f_2]=f_1\T t$, we obtain
$(f_1\T t)^{(n_1-1)+n_2+1}v=0$.

So again as in {\bf Case $1$} we are able to decompose the module $W_{s_{12}(\la)}$ into three subquotients of the form
$W_{\kappa(\la+\om_1)}$ (to be precise, the quotients of these modules).

Now one easily checks that the cases of $\sigma=s_1$ and $\sigma=s_2$ are equivalent to the case
$\sigma=s_{12}$.

\subsection{Type $C_2$}
The goal of this section is to prove Theorem \ref{stepofdecomposition} for $\fg$ of type $C_2$.
The longest element $w_0$ is equal to $-1$, so $t_{w_0\om_i}=t_{-\om_i}$.
We denote by $\alpha_1$ the short simple root, by $\alpha_2$ the long simple root,
$\Delta_+=\lbrace \alpha_1, \alpha_2, \alpha_2+\alpha_1, \alpha_2+2\alpha_1 \rbrace$ and the set of corresponding coroots is
$\lbrace \alpha_1^\vee, \alpha_2^\vee, 2\alpha_2^\vee+\alpha_1^\vee, \alpha_2^\vee+\alpha_1^\vee\rbrace$.
We have the following sequences of $\beta$'s:
\[\beta^1_1=-\alpha_1+\delta, \beta^1_2=-2\alpha_1-\alpha_2+\delta, \beta^1_3=-\alpha_1-\alpha_2+\delta;\]
\[\beta^2_1=-\alpha_2+\delta, \beta^2_2=-\alpha_1-\alpha_2+2\delta, \beta^2_3=-2\alpha_1-\alpha_2+\delta, \beta^2_4=-\alpha_1-\alpha_2+\delta.\]

The quantum Bruhat graph is shown on Figure \ref{QBGC2}.

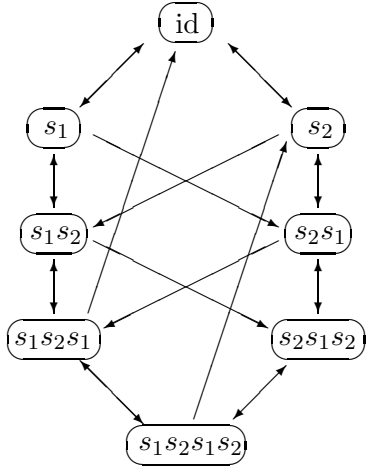
\begin{figure}\label{QBGC2}
\begin{picture}(200,200)

\put(100,180){\oval(20,15)}
\put(96,176){${\rm id}$}

\put(50,140){\oval(20,15)}
\put(46,138){$s_1$}

\put(150,140){\oval(20,15)}
\put(146,138){$s_2$}

\put(50,100){\oval(25,15)}
\put(41,98){$s_1s_2$}

\put(150,100){\oval(25,15)}
\put(141,98){$s_2s_1$}

\put(50,60){\oval(35,15)}
\put(35,58){$s_1s_2s_1$}

\put(150,60){\oval(35,15)}
\put(135,58){$s_2s_1s_2$}

\put(100,20){\oval(45,15)}
\put(82,18){$s_1s_2s_1s_2$}

\put(72,160){\vector(1,1){12}}
\put(72,160){\vector(-1,-1){12}}

\put(128,160){\vector(-1,1){12}}
\put(128,160){\vector(1,-1){12}}

\put(50,120){\vector(0,1){10}}
\put(50,120){\vector(0,-1){10}}

\put(150,120){\vector(0,1){10}}
\put(150,120){\vector(0,-1){10}}

\put(50,80){\vector(0,1){10}}
\put(50,80){\vector(0,-1){10}}

\put(150,80){\vector(0,1){10}}
\put(150,80){\vector(0,-1){10}}

\put(72,40){\vector(-1,1){12}}
\put(72,40){\vector(1,-1){12}}

\put(128,40){\vector(1,1){10}}
\put(128,40){\vector(-1,-1){10}}

\put(100,120){\vector(2,-1){35}}
\put(100,120){\line(-2,1){35}}

\put(100,80){\vector(2,-1){31}}
\put(100,80){\line(-2,1){35}}

\put(100,120){\vector(-2,-1){35}}
\put(100,120){\line(2,1){35}}

\put(100,80){\vector(-2,-1){31}}
\put(100,80){\line(2,1){35}}

\put(63,70){\vector(1,3){33}}

\put(103,30){\vector(1,3){35}}
\end{picture}
\caption{$QBG$ of type $C_2$}
\end{figure}

\begin{prop}\label{edgesC_2}
Let $\bar\beta^i$ be the sequence of $\beta$'s for some reduced decomposition of the element $t_{-\omega_i}$, $i=1,2$.
%We consider the module $W_{\sigma(\lambda)}(\bar\beta^i,m)$.
If there is no edge $\sigma \stackrel{{\rm Re} \beta_{m+1}}{\longrightarrow} \sigma s_{{\rm Re} \beta_{m+1}}$, then
\[W_{\sigma(\lambda)}(\bar\beta^i,m) \simeq W_{\sigma(\lambda)}(\bar\beta^i,m+1).\]
\end{prop}
\begin{proof}
Lemma \ref{edgesinQBG} tells us that we need to consider two cases. Assume that
there are elements $\tau, \eta \in \Delta_+$ such that:
\[\tau, \eta \neq -{\rm Re} \beta_{m+1},\]
\[\tau + \eta=2\frac{\langle \tau,
{\rm Re} \beta_{m+1} \rangle}{\langle {\rm Re} \beta_{m+1}, {\rm Re} \beta_{m+1} \rangle}{\rm Re} \beta_{m+1},\]
\[\widehat{\sigma}(\tau) + \widehat{\sigma}(\eta)=
2\frac{\langle \tau, {\rm Re} \beta_{m+1} \rangle}{\langle {\rm Re} \beta_{m+1}, {\rm Re} \beta_{m+1} \rangle}\widehat{\sigma}
{\rm Re} \beta_{m+1}.\]
Then $-{\rm Re} \beta_{m+1}$ is equal to $\alpha_2+\alpha_1$ or  $\alpha_2+2\alpha_1$. Assume that $-{\rm Re} \beta_{m+1}=\alpha_2+\alpha_1$.
Then $\tau=\alpha_1$, $\eta=\alpha_2$ or $\tau=2\alpha_1+\alpha_2$, $\eta=\alpha_2$. We work out the first case
(the second case can be done by a direct computation). In the first case
$e_{\widehat{\sigma}-{\rm Re} \beta_{m+1}}$ is an element of a Lie algebra with simple root vectors
$e_{\widehat{\sigma}(\alpha_1)}$, $e_{\widehat{\sigma}(\alpha_2)}$. Using Corollary \ref{order}, $iii)$ we have that
$l_{{\rm Re} \beta_{m+1},m}>l_{\alpha_2,m}+2l_{\alpha_1,m}.$ But using BGG resolution we obtain that
$\widehat \sigma(e_{{\rm Re} \beta_{m+1}})^{l_{\alpha_2,m}+2l_{\alpha_1,m}+1}v=0$.

Now assume that $-{\rm Re} \beta_{m+1}=\alpha_2+2\alpha_1$. Then $\tau=\alpha_1, \eta=\alpha_2+\alpha_1$.
If the subspace spanned by
$\widehat\sigma e_{\alpha_1},\widehat\sigma e_{\alpha_2+2\alpha_1},\widehat\sigma e_{\alpha_2+\alpha_1},\widehat\sigma e_{\alpha_2}$
is closed under the Lie bracket then we can analogously to the previous case use BGG resolution. Conversely, if
the subspace spanned by
\[
\widehat\sigma e_{\alpha_1},\widehat\sigma e_{\alpha_2+2\alpha_1},
\widehat\sigma e_{\alpha_2+\alpha_1},\widehat\sigma f_{-\alpha_2}\otimes t
\]
is closed under the Lie bracket, then the needed equation is equivalent to
\[
(\widehat\sigma f_{-\alpha_2}\otimes t)^{l_{\alpha_2,m}+l{\alpha_1,m}+1}
(\widehat\sigma e_{\alpha_2+\alpha_1})^{l_{\alpha_2+\alpha_1,m}+1}v=0.
\]

Now we assume that $\widehat\sigma({\rm Re} \beta_{m+1})\in \Delta_-$, $-{\rm Re} \beta_{m+1}=\alpha_1 + \alpha_2$.
Then the only situation not covered by the previous case is $\sigma=w_0$ (the longest element of the Weyl group).
Using Corollary \ref{order}, $iii)$ we have $l_{\alpha_1+\alpha_2}=l_{2\alpha_1+\alpha_2}+l_{\alpha_2}+1$.
But using a direct computation in the algebra spanned by
$f_{-2\alpha_1-\alpha_2}\otimes~t$, $f_{-\alpha_1-\alpha_2}\otimes~t$, $f_{-\alpha_1}\otimes~t, f_{-\alpha_2}\T t$ we obtain:
\[(f_{-\alpha_1-\alpha_2}\otimes t)^{l_{2\alpha_1+\alpha_2}+l_{\alpha_2}+1}v=0.\]

This completes the proof.
\end{proof}

\begin{prop}\label{decompositionC_2}
Let $\bar\beta^i$ be a sequence of $\beta$'s for some reduced decomposition of the element $t_{-\omega_i}$, $i=1,2$.
If there exists an edge $\sigma \stackrel{{\rm Re} \beta_{m+1}}{\longrightarrow} \sigma s_{{\rm Re} \beta_{m+1}}$,
then there exists a surjection
\[
W_{{\sigma}s_{{\rm Re} \beta_{m+1}}(\lambda)}(\bar\beta^i,m+1) \twoheadrightarrow
{\rm U}(\fn^{af}) \widehat{\sigma}(e_{-{\rm Re} \beta_{m+1}})^{l_{-{\rm Re} \beta_{m+1},m}} v.
\]
\end{prop}
\begin{proof}
We need to prove the following equalities:
\begin{equation}\label{reflectionequalityC2}
\widehat{\sigma s_{{\rm Re} \beta_{m+1}}}(e_{\tau})^{l_{\tau,m+1}}
\widehat{\sigma}(e_{-{\rm Re} \beta_{m+1}})^{l_{-{\rm Re} \beta_{m+1},m}} v =0.
\end{equation}

Note first that if both
${\rm Re} \beta_{m+1}$ and  $\tau$ are long roots then $\mathbb{Z}\langle {\rm Re} \beta_{m+1}, \tau \rangle\simeq A_1 \oplus A_1$.
Therefore we can use Proposition \ref{subsystemofrank2}.

For natural numbers $a, b, c$ such that $a+b+c \geq m_1+2m_2+1$ we prove the following equality:

\begin{equation}\label{stringinequality}
(\widehat\sigma(e_{\alpha_2}))^a(\widehat\sigma(e_{\alpha_2+\alpha_1}))^b(\widehat\sigma(e_{\alpha_2+2\alpha_1}))^c v=0.
\end{equation}

If $\sigma={\rm id}$ or $s_{\alpha_1+\alpha_2}$, then $\widehat\sigma(\fn_+)$ is isomorphic to $\fn_+$ and therefore the equality
is a consequence of the BGG resolution. Assume that $\sigma=s_{\alpha_1}$ or $w_0$. We proceed with the decreasing induction in $c$.
It is obvious that the equality holds for $c \geq m_1+m_2+1$ and for $b=0$. Assume that this equality holds for all $c > c_0$.
Then using that $\widehat\sigma(f_{-\alpha_1} \otimes t)=e_{\alpha_1}$ write:

\begin{multline*}0=e_{\alpha_1}\widehat\sigma(e_{\alpha_2}))^a(\widehat\sigma(e_{\alpha_2+\alpha_1}))^{b-1}
(\widehat\sigma(e_{\alpha_2+2\alpha_1}))^{c_0+1} v=\\
\widehat\sigma(e_{\alpha_2}))^{a+1}(\widehat\sigma(e_{\alpha_2+\alpha_1}))^{b-2}(\widehat\sigma(e_{\alpha_2+2\alpha_1}))^{c_0+1} v+\\
\widehat\sigma(e_{\alpha_2}))^a(\widehat\sigma(e_{\alpha_2+\alpha_1}))^{b}(\widehat\sigma(e_{\alpha_2+2\alpha_1}))^{c_0} v.
\end{multline*}
Thus the needed equality holds for $a, b, c_0$.

Now assume that $\sigma=s_{\alpha_2}$ or $\sigma=s_2 s_1$. Then multiplying the equality $e_{\alpha_1}^{m_1+2m_2+1}v=0$ to
$(e_{\alpha_1}\widehat\sigma(e_{\alpha_1+\alpha_2})^c$ we obtain the needed equation for $a=0$.
Then the needed relation is the equivalent to the relation
\[(\widehat\sigma(f_{-\alpha_1}\otimes t))^a(\widehat\sigma(e_{\alpha_2+\alpha_1}))^{a+b}(\widehat\sigma(e_{\alpha_2+2\alpha_1}))^c v=0.\]

Finally assume that $\sigma=s_{\alpha_1+\alpha_2}$ or $\sigma=s_1 s_2$. We will prove the needed equality by induction in $b$.
The equality is obvious for $b=0$. Assume that in holds for $b=b_0$. Then the needed equality is equivalent to
\[\widehat\sigma(f_{-\alpha_1}\otimes t)(\widehat\sigma(e_{\alpha_2}))^a(\widehat\sigma(e_{\alpha_2+\alpha_1}))^{b_0-1}
(\widehat\sigma(e_{\alpha_2+2\alpha_1}))^{c+1} v=0.\]

The equation $(\widehat\sigma(e_{\alpha_1}))^a(\widehat\sigma(e_{\alpha_2+2\alpha_1}))^bv=0$ for $a+b \geq m_1+m_2+1$ can be obtained in the similar way. Finally, all the equalities \eqref{reflectionequalityC2} are either equivalent to partial cases of
\eqref{stringinequality} and the last equality or can be easily
derived from them.
%It is easy to see that if $s_{{\rm Re} \beta_{m+1}}\tau \in \delta_-$ then the needed equation can be obtained by direct computation using the fact
%that for any $\eta \in \Delta_-$, $k \geq 0$:
%\[\left(\widehat{\sigma}(f_\eta)t^k\right)^{l_{\eta,m}+1}v=0.\]
\end{proof}

\subsection{Type $G_2$}
Let $\fg$ be the Lie algebra of type $G_2$. The QBG of type $G_2$ can be found in \cite{LL}, p.19, figure $2$.
Using Proposition \ref{descriptionbeta} we obtain the following sequences $\bar \beta^1, \bar \beta^2$:
\begin{multline}\label{betasG21}
{\beta_1^1}^\vee=-\alpha_1^\vee+ \delta,{\beta_2^1}^\vee=-\alpha_2^\vee-3\alpha_1^\vee+3 \delta, \\{\beta_3^1}^\vee=-\alpha_2^\vee-2\alpha_1^\vee+2\delta,
{\beta_4^1}^\vee=-2\alpha_2^\vee-3\alpha_1^\vee+3 \delta,\\
{\beta_5^1}^\vee=-\alpha_2^\vee-3\alpha_1^\vee+2 \delta, {\beta_6^1}^\vee=-\alpha_2^\vee-\alpha_1^\vee+\delta,
{\beta_7^1}^\vee=-2\alpha_2^\vee-3\alpha_1^\vee+2 \delta,\\
{\beta_8^1}^\vee=-\alpha_2^\vee-2\alpha_1^\vee+\delta, {\beta_9^1}^\vee=-\alpha_2^\vee-3\alpha_1^\vee+\delta,
{\beta_{10}^1}^\vee=-2\alpha_2^\vee-3\alpha_1^\vee+\delta.
\end{multline}
\begin{multline}\label{betasG22}
{\beta_1^2}^\vee=-\alpha_2^\vee+ \delta, {\beta_2^2}^\vee=-\alpha_2^\vee-\alpha_1^\vee+\delta,{\beta_3^2}^\vee=-2\alpha_2^\vee-3\alpha_1^\vee+2 \delta,\\
{\beta_4^2}^\vee=-\alpha_2^\vee-2\alpha_1^\vee+\delta,{\beta_5^2}^\vee=-\alpha_2^\vee-3\alpha_1^\vee+\delta,
{\beta_6^2}^\vee=-2\alpha_2^\vee-3\alpha_1^\vee+ \delta.
\end{multline}

The quantum Bruhat graph is the following.
There are Bruhat edges from any element of the length $p$ to any element of the length $p+1$, $0 \leq p \leq 5$. There is the quantum
edge from any element with the reduced decomposition $(\prod s_{i_k})s_j$ to the element $(\prod s_{i_k})$, $j, i_k \in \lbrace 1,2 \rbrace$,
from any element with the reduced decomposition $(\prod s_{i_k})s_2 s_1 s_2$ to the element $(\prod s_{i_k})$ and
from any element with the reduced decomposition $(\prod s_{i_k})s_1s_2 s_1 s_2s_1$ to the element $(\prod s_{i_k})$.

Using this data we obtain that $E_{-\omega_1}(1,1,0)=15$, $E_{-\omega_2}(1,1,0)=7$. On the other hand the dimensions of
fundamental modules are known, see Remark \ref{fw}: $\dim W(\omega_1)=15$, $\dim W(\omega_2)=7$.

\begin{prop}
Assume that there is no edge $w \stackrel{\alpha}{\longrightarrow} w s_{{{\rm Re} \beta^i_{m+1}}}$. Then we have:
\[W_{\sigma(\lambda)}(\bar\beta^i,m) \simeq W_{\sigma(\lambda)}(\bar\beta^i,m+1).\]
\end{prop}
\begin{proof}
Lemma \ref{edgesinQBG} tells us that we need to consider two cases. Assume that
there are elements $\tau, \eta \in \Delta_+$ such that:
\[\tau, \eta \neq (-{\rm Re} \beta^i_{m+1}),\]
\[\tau + \eta=2\frac{\langle \tau,
{\rm Re} \beta^i_{m+1} \rangle}{\langle {\rm Re} \beta^i_{m+1}, {\rm Re} \beta^i_{m+1} \rangle}{\rm Re} \beta^i_{m+1},\]
\[\widehat{\sigma}(\tau) + \widehat{\sigma}(\eta)=
2\frac{\langle \tau, {\rm Re} \beta^i_{m+1} \rangle}{\langle {\rm Re} \beta^i_{m+1}, {\rm Re} \beta^i_{m+1} \rangle}\widehat{\sigma}
{\rm Re} \beta^i_{m+1}.\]
Then
\begin{gather*}
\tau + \eta=-{\rm Re} \beta^i_{m+1},\\
\widehat{\sigma}(\tau) + \widehat{\sigma}(\eta)= \widehat{\sigma} (-{\rm Re} \beta^i_{m+1})
\end{gather*}
or
\begin{gather*}
\tau + \eta=-3{\rm Re} \beta^i_{m+1},\\
\widehat{\sigma}(\tau) + \widehat{\sigma}(\eta)=3\widehat{\sigma}(-{\rm Re} \beta^i_{m+1}).
\end{gather*}
We consider the first case (the second case can be done similarly by a direct computation).
Assume that $-{\rm Re} \beta^i_{m+1}=\alpha_1+\alpha_2$. Then $m>0$ and $l_{-{\rm Re}\beta^i_{m+1},m}>3l_{\alpha_1,m}+l_{\alpha_2,m}$.
But using BGG resolution we have:
\[(\widehat{\sigma}(e_{-{\rm Re} \beta^i_{m+1}}))^{3l_{\alpha_1,m}+l_{\alpha_2,m}+1}v=0.\]

Now assume that $-{\rm Re} \beta^i_{m+1}=\alpha_1+2\alpha_2$. Then $l_{-{\rm Re}\beta^i_{m+1},m}>l_{\alpha_1+\alpha_2,m}+l_{\alpha_2,m}$.
 If the set $[\widehat{\sigma}(e_{\alpha_2}),\widehat{\sigma}(e_{\alpha_1})]=\widehat{\sigma}(e_{\alpha_1+\alpha_2})$,
 then using BGG
resolution we have:
\[(\widehat{\sigma}(e_{-{\rm Re} \beta^i_{m+1}}))^{l_{\alpha_1+\alpha_2,m}+l_{\alpha_2,m}}v=0.\]

Conversely we have that $[\widehat\sigma(f_{-\alpha_1}\otimes t),\widehat\sigma(e_{\alpha_1+\alpha_2})]=
\widehat\sigma(e_{\alpha_2})$ and using this
fact we obtain:
\[\widehat{\sigma}(e_{-{\rm Re} \beta^i_{m+1}}))^{l_{\alpha_1+\alpha_2,m}+l_{\alpha_2,m}+1}v=0.\]

In the similar way we prove the claim for  $-{\rm Re} \beta^i_{m+1}=\alpha_1+3\alpha_2$ or $-{\rm Re} \beta^i_{m+1}=2\alpha_1+3\alpha_2$.

Now assume that there do not exist such $\tau$ and $\eta$ that $-{\rm Re}\beta^i_{m+1}$ is nonsimple short and
$\widehat \sigma -{\rm Re}\beta^i_{m+1} \in \Delta_+$.
Then the only possible cases are $\sigma=w_0$ or $\sigma=s_{2 \alpha_1+3 \alpha_2}$. Then using the direct computation we obtain
that $(\widehat \sigma e_{-{\rm Re}\beta^i_{m+1}})^{l_{-{\rm Re}\beta^i_{m+1},m}}$ lie in the left ideal generated by
$(\widehat \sigma e_{\alpha})^{l_{\alpha,m}}$,  $\alpha\neq -{\rm Re}\beta^i_{m+1}$.
\end{proof}

\begin{prop}
We consider a module $W_{\sigma(\Lambda)}(\bar\beta^i,m)$. If there exists an edge
\[\sigma \stackrel{{\rm Re}\beta^i_{m+1}}{\longrightarrow} \sigma s_{{{\rm Re} \beta^i_{m+1}}}\]
in the quantum Bruhat graph, then $U(\mathfrak{n}^{af}) \widehat \sigma(e_{-{{\rm Re} \beta^i_{m+1}}})^{l_{{\rm Re} \beta^i_{m+1},m}}v$ is the quotient module of
$W_{\sigma s_{{{\rm Re} \beta^i_{m+1}}}(\lambda)}$.
\end{prop}

\begin{proof}
Let $v_1=\widehat \sigma(e_{-{{\rm Re} \beta^i_{m+1}}})^{l_{-{\rm Re} \beta^i_{m+1},m}}v$.
If $\langle {\rm Re} \beta^i_{m+1},\eta \rangle=0$, then it is easy to see that $[f_{{\rm Re} \beta^i_{m+1}},f_\eta]=0$ and
thus $\mathbb{Z}\langle {\rm Re} \beta^i_{m+1},\eta \rangle\cap \Delta$ is the root system of type $A_1\oplus A_1$. Therefore
the claim is a consequence of the Lemma \ref{subsystemofrank2}.

If ${\rm Re} \beta^i_{m+1}$ is long, then for any long root $\eta \neq {\rm Re} \beta^i_{m+1}$ we have that
$\mathbb{Z}\langle {\rm Re} \beta^i_{m+1},\eta \rangle\cap \Delta$ is a root system of type $A_2$.
Indeed, the Lie algebra spanned by all long roots of $G_2$ is isomorphic to $A_2$.
Hence the claim is a consequence of the Lemma \ref{subsystemofrank2}.

We note that if $s_{{\rm Re} \beta^i_{m+1}}\eta \in \Delta_-$ then the needed relations can be obtained by the direct computation.

Now assume that ${\rm Re} \beta^i_{m+1}=\alpha_1$. Then $m=0$. Note that the cases of long $\eta$ or $\eta$ orthogonal to
$\alpha_1$  are already
covered. Let us prove the claim for $\eta=\alpha_2$ or $\eta=\alpha_1+ 3\alpha_2$.
If $\widehat \sigma (\alpha_1),\widehat \sigma (\alpha_2),\widehat \sigma (\alpha_1+\alpha_2)$ are
linear dependent then the claim is a consequence of the BGG resolution. Assume that $\widehat \sigma(\alpha_2) \in \Delta_+$. Then
$0=(\widehat \sigma f_{-\alpha_2}\otimes t)^{3m_1+2m_2+2}(\widehat \sigma e_{\alpha_1+3 \alpha_2})^{m_1+m_2+1}v=
\widehat \sigma e_{\alpha_1+\alpha_2}^{m_2+1} \widehat \sigma e_{\alpha_1}^{m_1}v$.
In the remaining case we have:
\[
0=(\widehat \sigma f_{-2 \alpha_1-3\alpha_2}\otimes t)^{m_1}(\widehat \sigma e_{\alpha_1+\alpha_2})^{3m_1+m_2+1}v=
\widehat \sigma e_{\alpha_1+\alpha_2}^{m_2+1}\widehat \sigma e_{\alpha_1}^{m_1}v.
\]

In the analogous way we prove that $(\widehat \sigma e_{\alpha_2})^{3m_1+m_2+1}(\widehat \sigma e_{\alpha_1})^{m_1}v=0$.

Now let us consider the case ${\rm Re}\beta^i_{m+1}=-2 \alpha_1-3\alpha_2$.
The only remaining cases (i.~e. cases of non-orthogonal to ${\rm Re}\beta^i_{m+1}$ and short $\eta$)
are $\eta=-\alpha_1-\alpha_2$ and $\eta=-\alpha_1-2\alpha_2$. We have
$l_{\alpha_1+\alpha_2}+l_{\alpha_1+2\alpha_2}=l_{2\alpha_1+3\alpha_2}-1$.
The proof in this case is straightforward.
For example for $\eta=\alpha_1+\alpha_2$:
\[\widehat\sigma(f_{-\alpha_1-\alpha_2}\otimes t)^{3m_1+m_2}\widehat\sigma(e_{2\alpha_1+3\alpha_2})^{2m_1+m_2}v=0\]
using $(\widehat\sigma (e_{\alpha_1+\alpha_2}\T t^k)^{l_{\alpha_1+\alpha_2}+1}=0$ and
$(\widehat\sigma (e_{\alpha_1}\T t^k)^{l_{\alpha_1}+1}=0$, $k \geq 0$.
Analogous straightforward proof works for ${\rm Re}\beta^i_{m+1}=-\alpha_1-3\alpha_2$.

Now we need to consider the case of short ${\rm Re}\beta^i_{m+1}$. If ${\rm Re}\beta^i_{m+1}=-\alpha_1-2\alpha_2$,
then the needed relations can be obtained straightforwardly by the direct computation.
Assume that ${\rm Re}\beta^i_{m+1}=-\alpha_2$. In this case $m=0$. Then the relation
$(\widehat \sigma e_{\alpha_1+3 \alpha_2})^{m_1+1}(\widehat \sigma e_{\alpha_2})^{m_2}v=0$ can be obtained using one of the two following arguments.
If
the roots $(\widehat \sigma e_{\alpha_1+3 \alpha_2}),(\widehat \sigma e_{\alpha_1}),(\widehat \sigma e_{\alpha_2})$ are linear dependent,
then we can use the BGG resolution.
If they are linear independent, then this relation is a consequence of the relation
\[(\widehat \sigma e_{\alpha_1+2 \alpha_2})^{m_1+1}(\widehat \sigma e_{\alpha_1+3 \alpha_2})^{m_1+m_2+1}v=0.\]
Independently of $\sigma$ using BGG resolution we can obtain:
\[(\widehat \sigma e_{\alpha_1})^{m_1+m_2+1}(\widehat \sigma e_{\alpha_2})^{m_2}v=0.\]
Two remaining relations can be obtained in the similar way.

For $-{\rm Re}\beta^i_{m+1}=\alpha_1+\alpha_2$ all relations can be obtained in a similar way.
\end{proof}

\appendix
\section{Cherednik-Orr conjecture for cominuscule weights}
Let $\la$ be a dominant weight and let $W(\la)$ be the corresponding Weyl module. In particular, $W(\la)$ is a cyclic module
over the algebra $\fn_-\T\bK[t]$. The PBW filtration $F_l$ on $W(\la)$ is defined as follows:
\[
F_l={\rm span} \{f_{\beta_1}\T t^{j_1}\dots f_{\beta_a}\T t^{j_a}v_\la, a\le l\}.
\]
The PBW character ${\rm ch}_{PBW} W_\la(x,q,s)$ is defined by the formula
\[
{\rm ch}_{PBW} W_\la(x,q,s)=\sum_{l\ge 0} s^l {\rm ch} F_l/F_{l-1}
\]
(for example, the term $e^\la$ corresponds to the cyclic vector $v_\la$). The Cherednik-Orr conjecture \cite{CO1}
says that
\[
{\rm ch}_{PBW} W(\la)(x,q,q)=w_0E_{w_0\la}(x,q^{-1},\infty).
\]
Since $E_{w_0\la}(x,q^{-1},\infty)=w_0{\rm ch} W_{\la}$, the conjecture can be stated in the form
${\rm ch}_{PBW} W(\la)(x,q,q)={\rm ch} W_{\la}$.
The conjecture has been proved in several special cases (see \cite{CF,FM1,FM2}).

A fundamental weight $\om_i$ is called cominuscule if the corresponding simple root $\al_i$ occurs with coefficient
one in the highest root. In other words, $\om_i$ is cominuscule if and only if the subalgebra of $\fn_-$ spanned by
$f_\beta$, such that $\langle \beta,\om_i\rangle<0$, is abelian. Here are the cominuscule weights: in type $A_n$ all the fundamentals,
in type $B_n$ only $\om_1$, in type $C_n$ only $\om_n$, in type $D_n$ three fundamentals $\om_1$, $\om_{n-1}$ and $\om_n$,
in type $E_6$ two fundamentals $\om_1$ and $\om_6$, in type $E_7$ only $\om_7$ (we use the standard Bourbaki
enumeration \cite{B}).

Our goal is to prove the following:
\begin{thm}
The Cherednik-Orr conjecture holds for the weights $\lambda=m\omega_i$ if $\om_i$ is cominuscule.
\end{thm}
\begin{proof}
Note that there exists $\sigma\in {\rm stab}(\la)\subset W$ such that the following two sets coincide:
\[\lbrace \Delta_-\cap \sigma\Delta_+\rbrace=
\lbrace \alpha \in \Delta_-|\langle \alpha, \omega_i \rangle=0 \rbrace.\]
Namely, let $I$ be the Dynkin diagram of $\fg$. Then $I\setminus i$ is the Dynkin diagram of a semisimple Lie algebra.
Then $\sigma$ is equal to the longest element of the Weyl group of this semisimple Lie algebra.

Thanks to Proposition \ref{welldefined} we have $W_{\lambda}=W_{\sigma(\lambda)}.$
Since $\om_i$ is cominuscule, the subalgebra ${\rm span}\{e_\alpha |\langle \alpha, \omega_i \rangle\neq 0\}$ is abelian.
Therefore
$\widehat{\sigma}(\fn_+)$ is closed under the Lie bracket.  Hence $\widehat{\sigma}$ induces an automorphism $\varphi$
of $\fn^{af}$ and the $\varphi$-twist of  $W(\la)$ is isomorphic to $W_{\sigma\la}$.
%We note that if $\al$ satisfies $\langle \al,\om_i\rangle >0$, then $\varphi (f_{-\al}\T t^k)=e_{-w_0\al}\T t^{k+1}$.
This gives the following
relation between the characters of $W(\la)$ and $W_{\la}$: if ${\rm ch} W(\la)=\sum_{\beta\in X_+} e^{\la-\beta} a_{\beta}(q)$
(for some polynomials $a_\beta(q)$ depending on $\beta\in X_+$), then
$${\rm ch} W_{\sigma(\la)}=\sum_{\beta} e^{\la-\beta} q^{d_i\langle\beta,\om_i\rangle}a_{\beta}(q),$$
where $d_i=\langle \al_i,\om_i\rangle^{-1}$ (so $d_i\langle\beta,\om_i\rangle$ is exactly the coefficient of $\al_i$ in $\beta$).
Now it suffices to note that the right hand side is equal to the PBW twisted character ${\rm ch}_{PBW} W({\la})|_{s=q}$.
Indeed, the module $W(\la)$ is generated from the cyclic vector by the action of the algebra
${\rm span}\{f_\alpha\T t^k |k\ge 0, \langle \alpha, \omega_i \rangle\neq 0\}$. Since $\om_i$ is cominuscule, for any negative root
$\al$ one has
$d_i\langle \alpha, \omega_i \rangle$ is either $-1$ or $0$. Therefore, the PBW degree of a weight $\la-\beta$, $\beta\in X_+$
vector in $W_\la$ is equal to $d_i\langle \beta,\om_i\rangle$.
\end{proof}

\section*{Acknowledgments}
We are grateful to I.Cherednik, M.Finkelberg, G.Fourier, D.Kus, F.Nomoto and D.Orr for useful remarks and references.
The research was supported by the grant RSF-DFG 16-41-01013.

\end{document}